\documentclass{article}
\usepackage{physics,amsthm,amscd,amssymb,verbatim,epsf,amsmath,amsfonts,mathrsfs,mathtools,graphicx,dirtytalk,pdfpages,mdframed,tikz-cd,thmtools,relsize,subcaption}
\usetikzlibrary{decorations.text,hobby,svg.path,patterns,angles}
\usetikzlibrary{decorations.pathreplacing}
\usetikzlibrary{arrows.meta}
\usepackage[colorlinks=true,linkcolor=blue,citecolor=blue]{hyperref}
\usepackage{cleveref}
\usepackage{caption}
\usepackage{subcaption}
\usepackage{geometry}
\usepackage[shortlabels]{enumitem}
\theoremstyle{plain}
\newtheorem{Thm}{Theorem}[section]
\newtheorem{Cor}[Thm]{Corollary}

\newtheorem{innernumthm}{Theorem}
\newenvironment{nThm}[1]
  {\renewcommand\theinnernumthm{#1}\innernumthm}
  {\endinnernumthm}

\newtheorem{innernumcor}{Corollary}
\newenvironment{nCor}[1]
  {\renewcommand\theinnernumcor{#1}\innernumcor}
  {\endinnernumcor}

\newtheorem{Lem}[Thm]{Lemma}
\newtheorem{Prop}[Thm]{Proposition}
\theoremstyle{definition}
\newtheorem{Def}[Thm]{Definition}

\declaretheorem[style=definition ,qed=$\triangle$,sibling=Thm, name=Example]{Ex}

\newtheorem{remark}[Thm]{Remark}
\theoremstyle{remark}
\numberwithin{equation}{section}

\newcommand{\td}{\mathrm{d}}
\newcommand{\slr}{\mathrm{SL}_2(\mathbb{R})}

\DeclareMathOperator{\sgn}{sgn}

\setlist[enumerate]{label=\arabic*.}

\date{\today}

\title{Properties and Transformations \\of Weingarten Surfaces}

\usepackage{authblk}
\author[1,*]{Brendan Guilfoyle}
\author[1,2,$\dagger$]{Morgan Robson}
\affil[1]{School of STEM\\
          Munster Technological University\\
          Tralee\\
          Co. Kerry\\
          Ireland.}
\affil[2]{Department of Computing and Mathematics\\
          South East Technological University, Waterford\\
          Ireland.}
\affil[*]{brendan.guilfoyle@mtu.ie}
\affil[$\dagger$]{Corresponding Author: morgan.robson@setu.ie}
\begin{document}
\maketitle
\begin{abstract}
This paper explores Weingarten relations satisfied by surfaces of revolution in Euclidean 3-space $\mathbb{E}^3$. Firstly, we establish that the local geometry of a surface around umbilic points restricts its possible Weingarten relations. We demonstrate that the rate at which the surface becomes spherical at umbilic points imposes bounds on the slope of any satisfied Weingarten relation, extending previous research by a number of authors.

Secondly, we investigate transformations between Weingarten relations through the action of $\slr$, acting as fractional linear transformations on the surface's curvatures. We integrate this action, which splits into three natural geometric actions on surfaces in $\mathbb{E}^3$, providing a method of generating rotationally symmetric solutions to a transformed Weingarten relation. This technique is applied to a class of Weingarten relations known as semi-quadratic. We prove the action is transitive on such relations and give a classification result on their solutions.



\end{abstract} 
\let\thefootnote\relax

\section{Introduction
} \label{sec: Introduction}

Introduced by J. Weingarten in 1861 \cite{w61}, Weingarten surfaces are a topic of classical differential geometry and have found applications in architectural design \cite{p.et21a, p.et21b, t.et19}. An oriented surface in Euclidean 3-space $\mathbb{E}^3$ is \textit{Weingarten} when its principal curvatures $k_1$ and $k_2$ satisfy a differentiable functional relationship expressed as
\begin{equation}\label{eq: intro weingarten relation}
W(k_1, k_2) = 0. 
\end{equation}
 $W$ is called the Weingarten relationship and is a non-linear second-order PDE satisfied by the surface. Recent work has focused on understanding how Weingarten relations determine geometric properties of their rotationally symmetric solutions \cite{ci22, fm22}. In conjunction with this theme, this paper investigates the possible Weingarten relations for surfaces of revolution. This is done through two approaches. Firstly, obstruction criteria for Weingarten relations are given in terms of local surface geometry around umbilic points. Secondly, given an initial surface of revolution and its Weingarten relation, a family surface transformations are applied, generating solutions to transformed Weingarten relationships. A classification result of certain Weingarten surfaces is then given. Further details are now given.
 
\subsubsection*{Obstructions to Weingarten Relations for Surfaces of Revolution}

Our first topic explores how the geometric behaviour of a surface near its umbilic points affects the supported Weingarten relations. This is done in terms of a surface's \textit{curvature diagram}, denoted as $\mathfrak{F}(\mathcal{S})$, which represents the set of curvatures $(k_1, k_2)$ attained by points on the surface $\mathcal{S}$ as a subset of the $k_1k_2$-plane. In the literature, $k_1$ and $k_2$ are typically labelled by the condition $k_2 \leq k_1$. A surface $\mathcal{S}$ is Weingarten with relation (\ref{eq: intro weingarten relation}) if and only if $\mathfrak{F}(\mathcal{S})\subseteq W^{-1}\{0\}$, hence $\mathfrak{F}(\mathcal{S})$ strongly determines the supported Weingarten relationships. Examples are depicted in Figure \ref{fig: curv space}. Points of $\mathcal{S}$ with equal principal curvatures are called \textit{umbilic points} and the diagonal $k_1 = k_2$ in the $k_1k_2$-plane is called the \textit{umbilic axis}. The study of umbilic points is a classical yet still active area of research \cite{g24, gk23, go23}. They are guaranteed to exist on closed $C^2$ surfaces of zero genus, and thus the curvature diagram of such surfaces must intersect the umbilic axis.
\begin{figure}[h!]
    \centering
    \begin{subfigure}{0.45\textwidth}
\begin{tikzpicture}[scale=1]
\draw [->] (-.5,0) -- (4,0);
\draw (4.2,0) node {$k_1$};
\draw [->] (0,-.5) -- (0,4);
\draw (0,4.2) node {$k_2$};
\draw (0.25,-0.25) node {O};
\draw [dashed,->] (-0.2,-0.2) -- (4,4);

\begin{scope}[xshift=25pt,yshift=25pt]
\coordinate (N) at (1-0.3,1-0.3);
\coordinate (O) at (1.8-0.3,1.5-0.3);
\coordinate (O2) at (2.4-0.3,2.1-0.8);
\coordinate (P) at (3-0.3,2.2-0.8);
\coordinate (P2) at (2.99-0.3,1.7-0.8);
\coordinate (Q) at (3.1-0.3,1.2-0.8);
\coordinate (Q2) at (2-0.3,1.05-0.8);
\coordinate (S) at (1.4-0.3,1.1-0.3);
\draw[fill=blue!20] plot [smooth] coordinates {(N) (O) (O2) (P) (P2) (Q) (S) (N)};
\draw (2.5-0.3,1.62-0.7) node {$\mathfrak{F}(\mathcal{S})$};
\draw[decoration={
            text along path,
            text={Umbilic Axis},
            text align={center},
            raise=0.2cm},decorate] (0,0) -- (3,3);
            
\end{scope}
            
\begin{scope}[xshift=10pt,yshift=10pt]
\draw[line width= 0.4mm,color=red,scale=1,domain=0:1.2,smooth,variable=\t]
  plot ({2.5+\t},{2.5+\t^3-\t^2+0.4*\t});
 \end{scope}
 
\begin{scope}[xshift=10pt,yshift=10pt]
\draw[line width= 0.4mm,color=orange,scale=1,domain=0.6:1.1,smooth,variable=\t]
  plot ({\t},{2*0.6-\t});
\end{scope}
\end{tikzpicture}
    \end{subfigure}
    \begin{subfigure}{0.45\textwidth}
    \centering 
    \begin{tikzpicture}[scale=1]
\draw [->] (-.5,0) -- (4,0);
\draw (4.2,0) node {$k_1$};
\draw [->] (0,-.5) -- (0,4);
\draw (0,4.2) node {$k_2$};
\draw (0.25,-0.25) node {O};
\draw [dashed,->] (-0.2,-0.2) -- (4,4);

\begin{scope}[xshift=25pt,yshift=25pt]
\coordinate (N) at (1-0.3,1-0.3);
\coordinate (O) at (1.8-0.3,1.5-0.3);
\coordinate (O2) at (2.4-0.3,2.1-0.8);
\coordinate (P) at (3-0.3,2.2-0.8);
\coordinate (P2) at (2.99-0.3,1.7-0.8);
\coordinate (Q) at (3.1-0.3,1.2-0.8);
\coordinate (Q2) at (2-0.3,1.05-0.8);
\coordinate (S) at (1.4-0.3,1.1-0.3);
\draw[fill=blue!20] plot [smooth] coordinates {(N) (O) (O2) (P) (P2) (Q) (S) (N)};
\draw (2.5-0.3,1.62-0.7) node {$\mathfrak{F}(\mathcal{S})$};
\draw[decoration={
            text along path,
            text={Umbilic Axis},
            text align={center},
            raise=0.2cm},decorate] (0,0) -- (3,3);
            
\draw [dashed,pattern=north west lines] (1-0.3,0.5-0.3) -- (1-0.3,1-0.3) -- (1.5-0.3,1-0.3);
\end{scope}
            
\begin{scope}[xshift=10pt,yshift=10pt]
\draw[line width= 0.4mm,color=red,scale=1,domain=0:1.2,smooth,variable=\t]
  plot ({2.5+\t},{2.5+\t^3-\t^2+0.4*\t});
\draw [dashed,pattern=north west lines] (2.5,2) -- (2.5,2.5) -- (3,2.5);
 \end{scope}
 
\begin{scope}[xshift=10pt,yshift=10pt]
\draw[line width= 0.4mm,color=orange,scale=1,domain=0.6:1.1,smooth,variable=\t]
  plot ({\t},{2*0.6-\t});
\draw [dashed,pattern=north west lines] (0.6,0.1) -- (0.6,0.6) -- (1.1,0.6);
\end{scope}
\end{tikzpicture}
    \end{subfigure}
    \caption{Left: 
    Pairs $(k_1,k_2)$ satisfying a CMC relationship $k_1+k_2=c$ (orange) and the curvature diagrams of a generic surface (blue) and a generic Weingarten surface (red). Right: The directions of negative slope from the umbilic axis.}
\label{fig: curv space}
\end{figure}
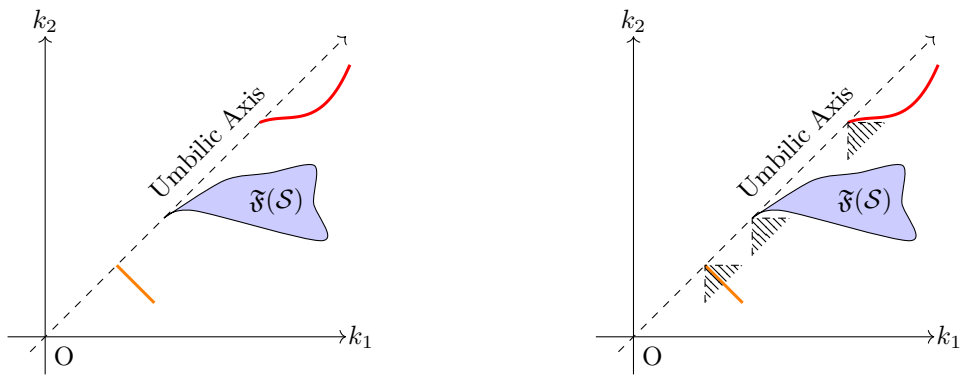
\\
Various authors have described the possible shapes of curvature diagrams, near the umbilic axis \cite{et99, Galvez22,hw54,Hopf1989a}. Different assumptions on the surface $\mathcal{S}$ are made, see \cite{et99}, Lemma 2 or \cite{Galvez22}, Theorem 1.1 for examples, however the general conclusion is that if $\mathfrak{F}(\mathcal{S})$ intersects the umbilic axis, it does so with a non-negative slope in the $k_1k_2$-plane, or, is a point on the umbilic axis (and thus $\mathcal{S}$ is congruent to a subset of the round sphere or a plane).
Directions of negative slope are shaded in Figure \ref{fig: curv space}. Due to the above, the orange curve in Figure \ref{fig: curv space} cannot be the curvature diagram of a surface, as it meets the umbilic axis from a direction of negative slope. Thus surfaces congruent to round spheres are the only surface homeomorphic to $S^2$ which can satisfy the relationship $k_1+k_2=c$, $c>0$, as is well known \cite{a62}.\\ \hfill \\
In this paper we give strictly positive lower bounds on the slope at which $\mathfrak{F}(\mathcal{S})$ intersects the umbilic axis for $\mathcal{S}$ a $C^2$-smooth surface of revolution. Throughout the rest of the paper, $\mathcal{S}$ will denote such a surface, in which case $\mathfrak{F}(\mathcal{S})$ is generically a curve (being the continuous image of the profile curve of $\mathcal{S}$). 
\begin{itemize}
    \item If $p\in\mathcal{S}$ is an umbilic point, the {\it umbilic slope at} $p$, denoted $\mu_p$, is the slope at which $\mathfrak{F}(\mathcal{S})$ meets the umbilic axis. We remark that $\mu_p$ may not be well defined for every $\mathcal{S}$.
    \item A surface is said to be \textit{totally umbilic} around $p$ if there exists a neighbourhood of $p$ in $\mathcal{S}$ which contains only umbilic points.
    \item $\mathcal{S}$ will be called \textit{non-flat at the point} $q\in\mathcal{S}$ if $K(q)\neq0$, and \textit{non-flat} if it is non-flat at every point.
     \item $\mathcal{S}$ will be called \textit{convex at the point} $q\in\mathcal{S}$ if $K(q)\geq 0$, and \textit{convex} if it is convex at every point.
    \item $\mathcal{S}$ will be called \textit{strictly convex at the point} $q\in\mathcal{S}$ if $K(q)>0$, and \textit{strictly convex} if it is strictly convex at every point.
\end{itemize}
 We remark that our usage of the word `non-flat' here does not coincide with the idea of a surface being distinct from a plane, locally. A point is non-flat if and only if it is not a parabolic and planar point.\\ \hfill \\
We first consider an umbilics points which lie off the axis of rotational symmetry.  With a minor technical assumption on the umbilic $p$, we show that if $\mathcal{S}$ is non-flat at $p$, then if $\mathfrak{F}(\mathcal{S})$ has a tangent line at $p$, it must be vertical. (Theorem \ref{thm: slope is unbounded}). The behaviour of $\mu_p$ for an umbilic point $p$ lying on the axis of rotational symmetry is then considered. Let $r_1$ and $r_2$ be the radii of curvature of $\mathcal{S}$ and $\theta$ the angle formed between $\mathcal{S}$'s (oriented) axis of rotational symmetry and its oriented normal vector. We may assume without loss of generality that $\theta=0$ at $p$.
\begin{nThm}{\ref*{thm: general slope restriction}}
Let $\mathcal{S}$ be strictly convex at an isolated umbilic point $p$ on the axis of rotational symmetry. Suppose at $p$, $\mathcal{S}$ has an umbilic slope of $\mu_p\in\mathbb{R}$.
\begin{enumerate}[label=(\Alph*)]
\item If the radii of curvature satisfy
$\lim\limits_{ \theta \to 0}\left( \frac{r_2-r_1}{\sin^\alpha \theta}\right)=\gamma$ for some $\alpha\geq 0$, $\gamma\in\mathbb{R}$, then $\mu_p \geq \alpha+1$, with equality if $\gamma \neq 0$.
\item Conversely if $\mu_p>\alpha+1$ then $\lim\limits_{ \theta \to 0}\left( \frac{r_2-r_1}{\sin^\alpha \theta}\right)=0$.
\end{enumerate}
\end{nThm}
\vspace{0.1in}
\noindent \Cref{thm: general slope restriction} is a corollary of \Cref{thm: weak general slope restriction} which bounds the limit superior and limit inferior of the average slope of $\mathfrak{R}(\mathcal{S})$ near an umbilic point, covering cases where $\mu_p$ may not be well defined. As a corollary of \Cref{thm: general slope restriction} we show the following 
\vspace{0.1in}
\begin{nCor}{\ref*{cor: slope if c4 or 4}}
Let $p$ be an isolated umbilic point on the axis of rotational symmetry of $\mathcal{S}$, a strictly convex and $C^3$-smooth surface. Then if $\mu_p$ exists, $\mu_p\geq 2$. If in addition $\mathcal{S}$ is $C^4$-smooth then $\mu_p\geq 3$.
\end{nCor}
\noindent The above theorems hence characterise $\mu_p$ as a measure of the rate at which $\mathcal{S}$ becomes umbilic.
\subsubsection*{$\slr$ Transformations}
Our second topic concerns a family of transformations which, for each Weingarten relation (\ref{eq: intro weingarten relation}) having a rotationally symmetric solution, produces a new relation which also admits a rotationally symmetric solution. $\slr$ acts on the $k_1k_2$-plane by real fractional linear transformations, which coincide with the isometries of the geometrised $k_1k_2$-plane considered in \cite{gk05}. Our main theorem is
\vspace{-0.5cm}
\begin{nThm}{\ref*{thm: new surface form old with sl2r}}
If $T$ is a real fractional linear transformation of the $k_1k_2$-plane with $\mathcal{S}$ non-flat, then there exists a rotationally symmetric and possibly non-regular surface $\widetilde{\mathcal{S}}$ such that $\mathfrak{F}(\widetilde{\mathcal{S}})=T(\mathfrak{F}(\mathcal{S}))$.
\end{nThm}
\noindent The action of $\slr$ is then described geometrically in terms of transformations of $\mathcal{S}$ in $\mathbb{E}^3$ (\Cref{thm: slr decomposition}). A class of surfaces called \textit{semi-quadratic Weingarten surfaces} satisfying a Weingarten relation of the form
\begin{equation}\label{intro eq: quadratic Weingarten relationship}
    \alpha k_1k_2+\beta k_1 + \gamma k_2 +\delta =0 \qquad \alpha,\beta,\gamma,\delta \in \mathbb{R},
\end{equation}
are investigated. This class contains well-known subclasses of surfaces, such as ones linear in $k_1$ and $k_2$, ones linear in the mean and Gauss curvature, $H$ and $K$, and ones linear in the radii of curvature, $r_1=\frac{1}{k_1}$ and $r_2=\frac{1}{k_2}$, investigated in \cite{lopez08,pampano20}, \cite{gmm03,lopez08 hyp} and \cite{gk21}, respectively. The quantities
\begin{align*}
&\Lambda_1=\beta-\gamma, &\Lambda_2=(\beta+\gamma)^2-4\alpha\delta,
\end{align*}
 are introduced which characterise when the PDE (\ref{intro eq: quadratic Weingarten relationship}) is elliptic, namely $\Lambda_2>\Lambda_1^2$ (\Cref{prop: qw ellipticity condition}).
When $\Lambda_1=0$ semi-quadratic surfaces become LW-surfaces. LW-surfaces are classified into three types, \textit{elliptic} when $\Lambda_2>0$ \cite{gmm03}, \textit{hyperbolic} when $\Lambda_2<0$ \cite{lopez08 hyp} and a border case $\Lambda_2=0$ which describe subsets of spheres, tubular surfaces or planes. This motivates a generalisation of the nomenclature:
\vspace{-0.1in}
\begin{Def}
    A semi-quadratic Weingarten surface satisfying $\Lambda_2>\Lambda_1^2$ is said to be \textit{elliptic}. If $\Lambda_2<\Lambda_1^2$ it is said to be \textit{hyperbolic}.
\end{Def}
\noindent Semi-quadratic relations form an invariant set under the action of $\slr$ and the ratio $\Lambda_1^2/\Lambda_2$ is shown to be an invariant (\Cref{prop: slr preserve qw}). The $\slr$ transformations are then shown to be transitive on all semi-quadratic relations satisfying $\Lambda_2>0$ and sharing the same invariant (\Cref{prop: sl2 acts transitively on QW}), therefore such semi-quadratic surfaces can be transitively related by induced transformations in $\mathbb{E}^3$. This is used to show the following.
\begin{nThm}{\ref*{thm: parabolic QW are canal}}
    Let $\mathcal{S}$ be a connected rotationally symmetric semi-quadratic Weingarten surface for which $\Lambda_1^2=\Lambda_2$. Then $\mathcal{S}$ is a subset of a round sphere, tubular surface or plane.
\end{nThm}
\begin{nThm}{\ref*{thm: slr on QW to pure lin}}
Any non-flat rotationally symmetric, connected semi-quadratic Weingarten surface with $\Lambda_2>0$ is the image under a composition of homotheties, parallel translations and reciprocal transformations of a Weingarten surface satisfying the relation
\begin{equation}\label{intro: pure k lin equ}
k_2=\lambda k_1,
\end{equation}
for $\lambda>0$ when the surface is elliptic, or for $\lambda<0$ when the surface is hyperbolic.
\end{nThm}
\noindent Surfaces of revolution satisfying relation (\ref{intro: pure k lin equ}) were classified in \cite{pampano20}. \Cref{thm: slr on QW to pure lin} therefore extends this classification to rotationally symmetric semi-quadratic surfaces with $\Lambda_2>0$.\\

\noindent The paper is organised as follows. Section \ref{sec: background} details our method of describing surfaces of revolution and introduces the radius of curvature equivalent of the curvature diagram, termed the {\it RoC diagram}. Section \ref{sec: rot sym weingarten} explores the possible RoC diagrams for surfaces of revolution while Section \ref{sec: slr trans} investigates the effect of the $\slr$ mappings on Weingarten relations and the transformations they induce on surfaces.
\vspace{0.1in}
\section{Background} \label{sec: background}
\subsection{The Curvature of Surfaces of Revolution.}
In this short subsection we define a coordinate system on $\mathcal{S}$ in the special case $\mathcal{S}$ is non-flat. By continuity of the Gauss curvature, if a $C^2$-smooth surface is non-flat at a point it is also non-flat in a neighbourhood of that point - hence the constructed coordinates will be used to describe surfaces of revolution locally around non-flat points. Position $\mathcal{S}$ in $\mathbb{R}^3$ with the axis of rotational symmetry aligned with the $z$ axis. The principal foliations of $\mathcal{S}$ are given by the parallels and profile curves of $\mathcal{S}$ whose respective principal curvatures we denote by $k_1$ and $k_2$. Note because $\mathcal{S}$ is assumed non-flat, both $k_1$ and $k_2$ are non-zero. Let $\alpha$ be the profile curve of $\mathcal{S}$ which lies in the $yz$-plane (i.e. the generating curve of $\mathcal{S}$). Since $k_2\neq 0$ the Gauss map $\mathcal{N}:\alpha \to S^1$ is a local diffeomorphism and we may thus use the \textit{Gauss angle} $\theta\in(-\pi,\pi]$ on $S^1$ to locally parameterise $\alpha$. Here $\theta$ is the angle made between the normal vector of $\alpha$, denoted $\hat{n}$, and the positive $z$ axis. We orient $\alpha$ so that $\hat{n}(\theta)=(\sin(\theta),\cos(\theta))$. Points of $\alpha$ on the $z$ axis such that $\theta=0$ or $\theta=\pi$ will be called \textit{north} and \textit{south} poles respectively. If $\rho$ and $h$ denote the respective $y$ and $z$ components of $\alpha$, they satisfy the relationship
\begin{equation}\label{eq: rho and h slope is -tan}
\dv{h}{\rho}=-\tan\theta,
\end{equation}
and $\mathcal{S}$ may be described via $\vec{X}(\theta,\phi)=(\rho(\theta)\sin\phi,\rho(\theta)\cos\phi,h(\theta))$ for $(\theta,\phi)\in(-\pi,\pi]\times[0,\pi)$. This is illustrated in Figure \ref{fig: profile curv and surface}.
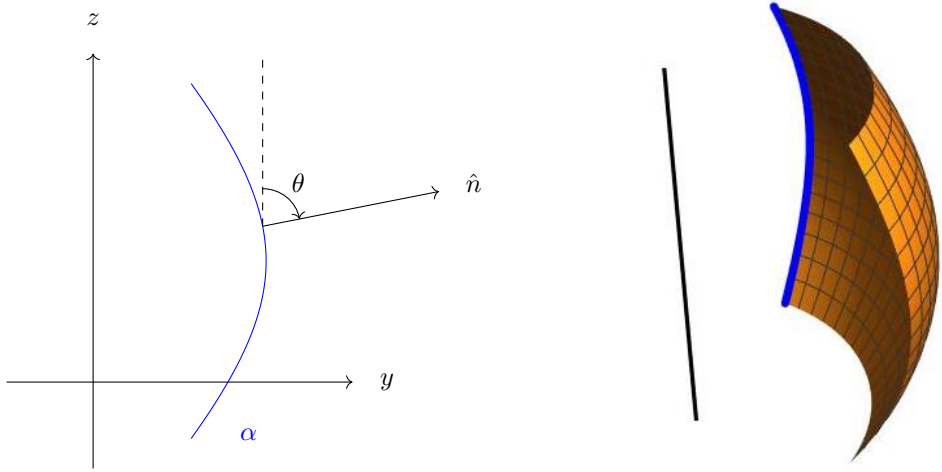
\begin{figure}[h!]
    \centering
    \begin{subfigure}{0.4\textwidth}
        \begin{tikzpicture}[scale=2.3]
\draw [->] (0,-1.2) -- (0,1.2);
\draw (0,1.2)++(0.0,0.2) node {$z$};
\draw [->] (-0.5,-.7) -- (1.5,-.7);
\draw (1.5,-.7)++(0.2,0.0) node {$y$};
\draw[color=blue,domain=-0.9:0.9,smooth,variable=\x]
  plot ({2-cosh(\x)},{sinh(\x)});
\draw (0.9,-1) node {$\color{blue}\alpha$};
\coordinate (P) at (0.979933244,0.201336003);
\draw[->] (0.979933244,0.201336003) -- ++(1.02006676,0.201336003) coordinate (n);
\draw (n) ++(0.2,0.05) node {$\hat{n}$};
\draw[dashed,-] (0.979933244,0.201336003) -- ++(0,1) coordinate (N);
\pic[draw, <-, "$\theta$", angle eccentricity=1.5] {angle = n--P--N};
\end{tikzpicture}
    \end{subfigure}
    \hspace{2cm}
    \begin{subfigure}{0.3\textwidth}
        \centering
        \includegraphics[width=0.9\textwidth]{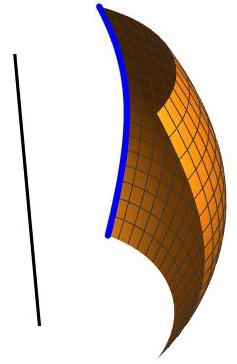} 
    \end{subfigure}
    \caption{The profile curve $\alpha$ (blue) in the $yz$-plane is parameterised by the angle $\theta$ and revolved by an angle of $\phi$ around the axis of rotation (black) to generate $\mathcal{S}$.} \label{fig: profile curv and surface}
\end{figure}

Throughout the paper we assume that $\mathcal{S}$ is fully revolved around the $z$-axis, so due to symmetry, if $\theta\in(-\pi,\pi)$ is a Gauss angle of $\mathcal{S}$ so is $-\theta$, with $\rho$ and $h$ satisfying $\rho(-\theta)=-\rho(\theta)$ and $h(-\theta)=h(\theta)$. Therefore we will often assume $\theta\in I$ where $I\subset [0,\pi]$ is the range of positive Gauss angles attained by $\mathcal{S}$. One may also describe $S$ by its \textit{support function}:
\[
r=\vec{X}\cdot \hat{N},
\]
with $\hat{N}=(\sin\theta\sin\phi,\sin\theta\cos\phi,\cos\theta)$ being the unit normal vector of $\mathcal{S}$. One can check that
\[
r=\rho(\theta)\sin\theta+h(\theta)\cos\theta
\] and so $r$ depends only on $\theta$ and $r(\theta)=r=(-\theta)$. We also have
\begin{align}\label{eq: rho, h in terms of r}
   &\rho=r\sin\theta  + \dv{r}{\theta}\cos\theta, 
   &h=r\cos\theta -\dv{r}{\theta}\sin\theta.
\end{align}
In this paper the radii of curvature of $\mathcal{S}$, namely $r_1:\mathcal{S}\to\mathbb{R}$ and $r_2:\mathcal{S}\to\mathbb{R}$, will be understood through their restrictions $r_1:\alpha\to\mathbb{R}$ and $r_2:\alpha\to\mathbb{R}$ to the profile curve $\alpha$. The corresponding coordinate expressions are denoted as $r_1(\theta)$ and $r_2(\theta)$ and by abuse of notation, often just as $r_1$ and $r_2$.
\begin{Prop}\label{prop: curvatures in terms of r}
The radii of curvature of $\mathcal{S}$ can be expressed as
\begin{align}\label{eq: curvatures in terms of r}
    &r_1=r+\dv{r}{\theta}\cot\theta, &r_2=r+ \dv[2]{r}{\theta},
\end{align}
for $\theta\in(0,\pi)$. Conversely, the support function $r$ and the coordinates $(\rho,h)$ are given by
\begin{equation}\label{eq: support from r1}
    r(\theta)=\frac{r(\theta_0)}{\cos\theta_0}\cos\theta+\cos \theta \int^\theta_{\theta_0}\frac{r_1(\theta)\sin\theta}{\cos^2\theta}\td \theta,
\end{equation}
\begin{align}\label{eq: rho in terms of r1}
    &\rho =r_1\sin\theta, &\dv{h}{\theta}=-r_2\sin\theta.
\end{align}
\end{Prop}
\begin{proof}
Firstly the pair of relations in (\ref{eq: rho in terms of r1}) are derived from equation (\ref{eq: rho and h slope is -tan}) and the standard formula for the curvatures of a surface of revolution in terms of their generating curve \cite[p.120]{Woodward2019}. Relations (\ref{eq: curvatures in terms of r}) then follow from equations (\ref{eq: rho, h in terms of r}) and (\ref{eq: rho in terms of r1}). Finally equation (\ref{eq: support from r1}) follows from integration of the equation for $r_1$ in (\ref{eq: curvatures in terms of r}).
\end{proof}

\noindent We remark that as a corollary of equations (\ref{eq: rho in terms of r1}), $\rho$ takes the same sign as $r_1$ when $\theta\in[0,\pi]$ and
\begin{align*}
    &r_1(-\theta)=r_1(\theta) &&\text{and} &r_2(-\theta)=r_2(\theta).
\end{align*}
Thus whenever $r_1$ and $r_2$ are differentiable functions at $\theta=0$ we have $r_1'(0)=r_2'(0)=0$.\\ \hfill \\

In the general scenario, the radii of curvature of a $C^3$-smooth surface cannot be assumed differentiable at an isolated umbilic point. However in the case considered in this section, where $\mathcal{S}$ is rotationally symmetric and strictly convex, the regularity of the radii of curvatures is implied by the regularity of $\mathcal{S}$.
\begin{Lem}\label{lem: C3 surf gives C1 radii}
    If $\mathcal{S}$ is a strictly convex, rotationally symmetric and $C^3$-smooth surface, then $r_1(\theta)$ and $r_2(\theta)$ are $C^1$-smooth on $[0,\pi]$.
\end{Lem}
\begin{proof}
    If $\mathcal{S}$ is $C^3$-smooth then the support function $r$ is also $C^3$-smooth. The claimed regularity of $r_2(\theta)$ follows trivially from equations (\ref{eq: curvatures in terms of r}). To show the claimed regularity of $r_1(\theta)$, note that for $\theta\in(0,\pi)$, equations (\ref{eq: curvatures in terms of r}) show $r_1(\theta)$ has regularity $C^{2}$, thus $C^{1}$ regularity is implied on $(0,\pi)$. We only need to show the required regularity at $\theta=0$ and $\theta=\pi$. We show the $\theta=0$ case first. Write by differentiation of equations (\ref{eq: curvatures in terms of r})
\begin{equation}\label{eq: r_1 first der}
    r_1'=\frac{r''\sin\theta\cos\theta-r'\cos^2\theta}{\sin^2\theta}.
\end{equation}
The numerator of this quotient vanishes as $\theta\to 0$ since $r'(0)=0$. Thus a straightforward application of L'Hopitals rule gives
\[
\lim_{\theta\to 0}r_1'(\theta)=\frac{1}{2}r'''(0)+r'(0)=0,
\]
since $r'''(0)=0$ also. The claim at $\theta=\pi$ follows similarly.
\end{proof}
\begin{Lem}\label{lem: C4 surf gives C2 radii}
    If $\mathcal{S}$ is a strictly convex, rotationally symmetric and $C^4$-smooth surface, then $r_1(\theta)$ and $r_2(\theta)$ are $C^2$-smooth on $[0,\pi]$.
\end{Lem}
\begin{proof}
    This lemma follows almost identically to the proof of the previous lemma, except this time we must check if $\lim\limits_{\theta\to0}r_1''(\theta)$ exists. Differentiating equation (\ref{eq: r_1 first der}) yields for all $\theta\in(0,\pi)$;
    \[
    r_1''=\frac{r'''\cos\theta\sin^2\theta-r''\sin\theta(1+\cos^2\theta)+2r'\cos(\theta)}{\sin^3\theta}.
    \]
    An application of L'Hopital yields
    \[
    \lim\limits_{\theta\to 0}r_1''(\theta)=\lim\limits_{\theta\to 0}\left(\frac{1}{3}r^{(4)}-\frac{2r'''\sin\theta}{3\cos\theta}+r''-\frac{2}{3}r'\cot\theta\right).
    \]
Note one quickly checks $r'\cot\theta \to r''(0)$ as $\theta \to 0$ giving
\[
 \lim\limits_{\theta\to 0}r_1''(\theta)=\frac{1}{3}\left(r^{(4)}(0)+r''(0)\right).
\]
\end{proof}
\begin{Prop}\label{prop: curve and CM iff surface}
The radii of curvature of $\mathcal{S}$ satisfy an integrability condition called the {\em derived Codazzi-Mainardi equation}
\begin{equation}
\label{eq: differentiated Codazzi-Mainardi}
\dv{r_1}{\theta}=(r_2-r_1)\cot \theta, \qquad \theta \in I\backslash\{0,\pi\}.
\end{equation}
Conversely, (\ref{eq: differentiated Codazzi-Mainardi}) is sufficient for a continuous map $\mathcal{J}:I \to \mathbb{R}^2\backslash\{(0,0)\}$, $\theta \mapsto (r_1(\theta),r_2(\theta))$, with $I\subset [0,\pi]$ and $r_1(\theta)$ differentiable except possibly at $\theta=0$ or $\pi$, to parameterise the radii of curvature of a rotationally symmetric $C^2$-smooth surface.
\end{Prop}
\begin{proof}
To derive the Codazzi-Mainardi equation, multiply the difference between equations (\ref{eq: curvatures in terms of r}) by $\cot\theta$ and integrate between $\theta_1$ and $\theta_2$, for $(\theta_1,\theta_2)\subset(0,\pi)$, to derive the integral relationship
\begin{equation}\label{eq: integrated cm}
r_1(\theta_2)-r_1(\theta_1)=\int^{\theta_2}_{\theta_1}(r_2-r_1)\cot \theta \td\theta.
\end{equation}
This is referred to as the \textit{integrated Codazzi-Mainardi relationship}. Dividing by $\theta_2-\theta_1$ and letting $\theta_2 \to \theta_1$ yields the result. Conversely, if $\mathcal{J}(\theta)=(r_1(\theta),r_2(\theta))$ is as stated, then define a $C^2$-smooth function $r:I\to\mathbb{R}$ by
\[
r_2(\theta)=r+\dv[2]{r}{\theta}.
\]
From the Codazzi-Mainardi equation it is easy to show $r$ satisfies
\[
r_1(\theta)=r+\cot(\theta)\dv{r}{\theta},
\]
for $\theta \in I$. Letting $\vec{X}=(\rho \sin\phi,\rho\cos\phi,h)$ where $\rho$ and $h$ are given by equation (\ref{eq: rho, h in terms of r}) gives a parametrisation for a $C^2$ rotationally symmetric surface whose support function is $r$ (given by equation (\ref{eq: support from r1}) explicitly) with the functions $r_1(\theta)$ and $r_2(\theta)$ as its radii of curvature. Since $r_1\neq 0$ and $r_2\neq 0$ by assumption, it can be quickly checked via equation (\ref{eq: rho in terms of r1}) that $\vec{X}_\theta\times\vec{X}_\phi\neq0$ and so the surface is regular.
\end{proof}
The equations presented in this section also allow us to describe \textit{non-regular surfaces}, i.e. images of $C^2$ homeomorphisms from an open subset of $\mathbb{R}^2$ to an open subset of $\mathbb{R}^3$ whose tangent vectors are linearly dependent at points. In present setting, one can check that $\vec{X}_\theta=0$ iff $r_2=0$ and $\vec{X}_\phi=0$ iff $r_1=0$, we call points with at least one vanishing radii of curvature \textit{cusps}.
\vspace{0.1in}
\subsection{The Radii of Curvature Diagram}
We will now again permit $\mathcal{S}$ to be possibly flat (i.e. to have vanishing Gauss curvature) at points. The curvature diagram $\mathfrak{F}(\mathcal{S})$ of $\mathcal{S}$ has already been introduced. We now introduce its radii of curvature equivalent for surfaces of revolution. Let $r_2$ be the radius of curvature of the profile curve of $\mathcal{S}$, and $r_1$ the radius of curvature of curves of constant $\theta$. Let $k_2$ and $k_1$ be their respective reciprocals.
\begin{Def}
The \textit{radii of curvature (RoC) diagram} of $\mathcal{S}$, $\mathfrak{R}(\mathcal{S})$, is the set
\[
\mathfrak{R}(\mathcal{S})=\left\{\left. (r_1,r_2)\in \widehat{\mathbb{R}} \cross \widehat{\mathbb{R}} \enskip \right |\enskip r_1,r_2 \text{ are radii of curvature attained at a point of }\mathcal{S}\right\},
\]
where $\widehat{\mathbb{R}}$ is the projectively extended real line, so $\widehat{\mathbb{R}} \cross \widehat{\mathbb{R}} \cong S^1\cross S^1$.
\end{Def}
\noindent The ambient space $\widehat{\mathbb{R}} \cross \widehat{\mathbb{R}}$ is called \textit{RoC space} and compactifies $\mathbb{R}^2$ by gluing in two copies of $S^1$ along the lines $r_1=\infty$ and $r_2=\infty$. The umbilic axis remains the diagonal line and surfaces that have zero Gauss curvature at points have RoC diagrams which are unbounded.
\begin{Def}
    Given a (possibly non-regular) surface $\mathcal{S}$, a point $p\in\mathcal{S}$ is said to be an \textit{umbilic point} of $\mathcal{S}$ if the pair $(r_1|_p,r_2|_p)$ lies on the diagonal of $\widehat{\mathbb{R}} \cross \widehat{\mathbb{R}}$, where $r_i|_p$, $i=1,2$ are the radii of curvature of $\mathcal{S}$ at $p$.
\end{Def}
\noindent Hence these definitions account for umbilics at flat points, when both $r_1$ and $r_2$ are non-finite, and also umbilics at cusps, when both $r_1$ and $r_2$ are
zero. To study the slope of $\mathfrak{R}(\mathcal{S})$ at the point $(r_0,r_0)$, with $r_0<\infty$ we will
consider the limiting value of the function $\mu(\theta)$;
\begin{equation}\label{eq: avg slope}
\mu(\theta)=\frac{r_2(\theta)-r_0}{r_1(\theta)-r_0},
\end{equation}
which gives the average slope of $\mathfrak{R}(\mathcal{S})$ between the points $(r_1(\theta),r_2(\theta))$ and $(r_0,r_0)$.
\begin{figure}[h!]
    \begin{subfigure}{0.45\textwidth}
    \centering
    \begin{tikzpicture}
\draw [->] (-.5,0) -- (4,0);
\draw (4.2,0) node {$r_1$};
\draw [->] (0,-.5) -- (0,4);
\draw (0,4.2) node {$r_2$};
\draw (-0.4,-0.4) node {O};
\draw [dashed,->] (-0.2,-0.2) -- (3.5,3.5);
\begin{scope}[xshift=5pt, yshift=5pt]
\draw[decoration={
            text along path,
            text={Umbilic Axis},
            text align={center},
            raise=-0.4cm},decorate] (0,0) -- (3,3);
\end{scope}
\begin{scope}[xshift=15pt, yshift=15pt]
\draw[line width= 0.4mm,color=red,scale=1,domain=2.5:3.5,smooth,variable=\t]
  plot (\t/3+5/3,\t);
 \end{scope}
\begin{scope}
\draw[line width= 0.4mm,color=orange,scale=1,domain=0.5:2,smooth,variable=\t]
  plot (\t,{\t/(4*\t-1)});
\end{scope}
\end{tikzpicture}
    \end{subfigure}
   \begin{subfigure}{0.45\textwidth}
    \centering
    \begin{tikzpicture}
\draw [->] (-.5,0) -- (4,0);
\draw (4.2,0) node {$r_1$};
\draw [->] (0,-.5) -- (0,4);
\draw (0,4.2) node {$r_2$};
\draw (-0.4,-0.4) node {O};
\draw [dashed,->] (-0.2,-0.2) -- (3.5,3.5);
\begin{scope}[xshift=5pt, yshift=5pt]
\draw[decoration={
            text along path,
            text={Umbilic Axis},
            text align={center},
            raise=-0.4cm},decorate] (0,0) -- (3,3);
\end{scope}

\begin{scope}[xshift=15pt, yshift=15pt]
\draw[line width= 0.4mm,color=red,scale=1,domain=2.5:3.5,smooth,variable=\t]
  plot (\t/3+5/3,\t);
\draw [dashed,pattern=north west lines] (2.5,2) -- (2.5,2.5) -- (3,2.5);
\draw [dashed,pattern=north west lines] (2,2.5) -- (2.5,2.5) -- (2.5,3);
 \end{scope}
\begin{scope}
\draw[line width= 0.4mm,color=orange,scale=1,domain=0.5:2,smooth,variable=\t]
  plot (\t,{\t/(4*\t-1)});
\draw [dashed,pattern=north west lines] (.5,0) -- (.5,.5) -- (1,.5);
 \draw [dashed,pattern=north west lines] (0,.5) -- (.5,.5) -- (.5,1);
\end{scope}
\end{tikzpicture}
    \end{subfigure}
    \caption{Left: The RoC diagram of 
    a Weingarten surface satisfying the Weingarten relation $r_2=3r_1-5$ (red) and pairs $(r_1,r_2)$ satisfying a CMC relationship $k_1+k_2=c$ (orange). Right: Directions of negative slope at points on the umbilic axis.}
\label{fig: roc space}
\end{figure}
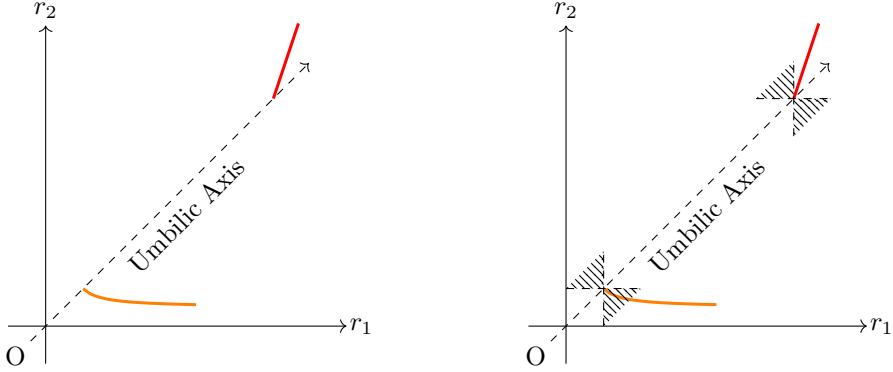
%
%

\begin{Def}
Let $p\in\mathcal{S}$ be a point of $\mathcal{S}$ with Gauss angle $\theta_0$. If $r_1(\theta_0)=r_2(\theta_0)=r_0$ so that $p$ is an umbilic point, we define the umbilic slope of $\mathcal{S}$ at $p$ as
\begin{equation}\label{eq: umbilic slope def}
    \mu_{p}=
\lim\limits_{\theta\to\theta_0}\left(\frac{r_2(\theta)-r_0}{r_1(\theta)-r_0}\right).
\end{equation}
\end{Def}
\noindent Note when $r_0$ is finite, the above definition of $\mu_p$ gives the slope of the tangent line to $\mathfrak{R}(\mathcal{S})$ in the $r_1r_2$-plane at the point $(r_0,r_0)$, when such a tangent line exists. When in addition $r_0\neq 0$, $\mu_p$ is equal to the slope of $\mathfrak{F}(\mathcal{S})$ at $(k_0,k_0)$ in the $k_1k_2$-plane:
\begin{equation}\label{eq: umbilic slope def with k}
    \mu_{p}=\lim\limits_{\theta\to \theta_0}\left(\frac{k_2(\theta)-k_0}{k_1(\theta)-k_0}\right),
\end{equation}
where $k_0=1/r_0$.
\hfill\\ \\ The limit (\ref{eq: umbilic slope def}) is, in general, not well defined for all choices of $\mathcal{S}$ or for all choices of $p\in\mathcal{S}$.  For example when $\mathcal{S}$ is a subset of a round sphere, both $r_1$ and $r_2$ are equal to $r_0$ at every point. It is also possible that umbilic points of $\mathcal{S}$ accumulate at $p$, leading $\mu(\theta)$ to be ill-defined on any open neighbourhood of $\theta_0$. In Section \ref{sec: rot sym weingarten} we will make assumptions on $\mathcal{S}$ around $p$ to avoid this scenario.  We remark that when a Weingarten relationship is given, namely if $W(k_1(\theta),k_2(\theta))=0$ in a punctured neighbourhood of $\theta_0$, then $\mu_p$ can be given in terms of the relationship $W$;
\begin{equation}\label{eq: umbilic slope in terms of WR}
    \mu_p=-\lim_{\theta\to \theta_0}\left({\pdv{W}{k_1}}\bigg / {\pdv{W}{k_2}} \right),
\end{equation}
\noindent a standard condition sufficient for the existence of $\mu_p$ is therefore that the gradient of $W(k_1,k_2)$ is non-vanishing at $p$ so the above limit is well defined \cite{Hopf1951,Hopf1989a}. A Weingarten surface is said to be \textit{elliptic}, if equation (\ref{eq: intro weingarten relation}) is an elliptic PDE. It can be shown \cite{Hopf1989a} that equation (\ref{eq: intro weingarten relation}) is elliptic at $q\in \mathcal{S}$ if and only if
\begin{equation}\label{eq: ellipticity cond}
   \left.  \left(\pdv{W}{k_1} \cdot \pdv{W}{k_2}\right) \right|_q >0.
\end{equation}
In particular, for a relation elliptic at the umbilic point $p\in\mathcal{S}$ we have that both $\left. \pdv{W}{k_1}\right|_p$ and $\left. \pdv{W}{k_2}\right|_p$ are non-zero, hence by equation (\ref{eq: umbilic slope in terms of WR}), elliptic Weingarten relations always have $\mu_p<0$.
\section{Obstructions to Weingarten Relations for Surfaces of Revolution}\label{sec: rot sym weingarten}
 In this section consequences of rotational symmetry are derived in terms of $\mathfrak{R}(\mathcal{S})$. The class of Weingarten surfaces solving a given relation can be large, as exemplified by the Weierstrass-Enneper representation for minimal surfaces \cite{hm87}. The requirement of rotational symmetry restricts the possible solutions to a relation greatly and often allows one to explicitly integrate the Weingarten relation (via equation (\ref{eq: differentiated Codazzi-Mainardi})) to find a $1$-parameter family of rotationally symmetric solutions.
\begin{Ex}\label{ex: hopf sphere r1}
The linear Hopf surfaces are defined as surfaces satisfying the Weingarten relation 
\begin{equation}\label{eq: lin hopf surf}
r_2=\lambda r_1+C,
\end{equation}
for $\lambda,C \in \mathbb{R}$. They are the stationary solutions to the linear Hopf curvature flow studied in \cite{gk21} and \cite{gr22}. Note that radii of curvature pairs $(r_1,r_2)=(0,C)$ and $(r_1,r_2)=(-C/\lambda,0)$ satisfy the linear Hopf relation, and so surfaces satisfying this relationship can potentially be non-regular. Inserting the linear Hopf relation into the Codazzi-Mainardi equation derives the separable ODE for $r_1$:
\[
\dv{r_1}{\theta}=\left(\left(\lambda-1\right)r_1+C\right)\cot\theta,
\]
which is solved to give
\begin{equation}\label{eq: hopf rad of curv}
r_1(\theta)=\frac{C}{1-\lambda}+\frac{A_0\sin^{\lambda-1}\theta}{\lambda-1}, \qquad A_0\in \mathbb{R}.
\end{equation}
We can then use equation (\ref{eq: support from r1}) to recover the support function for $\theta\in(\theta_1,\theta_2)$:
\[
r(\theta)=\frac{C}{1-\lambda}+\left(r_\text{\tiny{Hopf}}(\theta_1)-\frac{C}{1-\lambda}\right)\frac{\cos\theta}{\cos\theta_1}+\frac{A_0\cos\theta}{\lambda-1}\int^\theta_{\theta_1}\frac{\sin^\lambda \theta}{\cos^2\theta}\td \theta.
\]
One can check that when $A_0>C$, linear Hopf surfaces possess cusps so are non-regular.
\end{Ex}
 Our first point of discussion is to explore how the Codazzi-Mainardi equation controls the behaviour of $\mathfrak{R}(\mathcal{S})$. For example, it is quickly observed from equation (\ref{eq: differentiated Codazzi-Mainardi}) that for $\theta\in[0,\pi/2]$ if $\mathfrak{R}(\mathcal{S})$ is above the umbilic axis, i.e. $r_2>r_1$, $r_1$ must be increasing, and vice-versa when $\mathfrak{R}(\mathcal{S})$ is below the axis. These roles are reversed when $\theta\in[\pi/2,\pi]$, illustrated in Figure \ref{fig: vert slope}.  Equation (\ref{eq: differentiated Codazzi-Mainardi}) will be used to relate the slope of $\mathfrak{R}(\mathcal{S})$ at the umbilic axis to the rate at which $r_2-r_1$ vanishes as one approaches an umbilic point on $\mathcal{S}$. $\mathcal{S}$ is necessarily convex around umbilic points and thus our arguments will be local in nature - taking place in a convex subset of $\mathcal{S}$ in which $\theta$ may be used to parameterise the radii of curvature.
 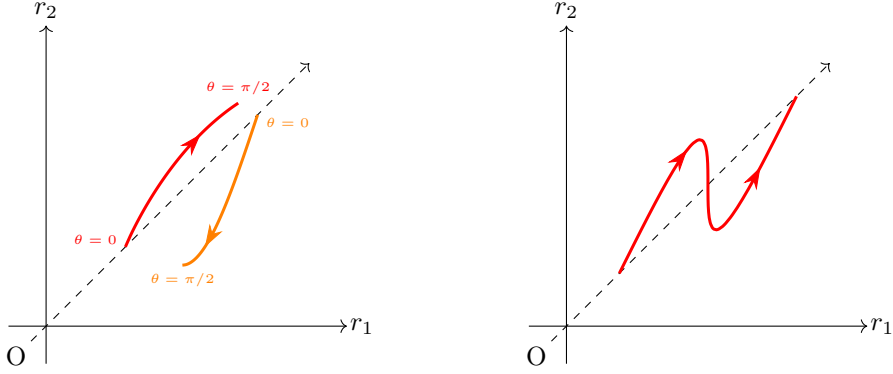
\begin{figure}[h!]
    \centering
    \begin{subfigure}{0.45\textwidth}
    \centering
    \begin{tikzpicture}
\draw [->] (-.5,0) -- (4,0);
\draw (4.2,0) node {$r_1$};
\draw [->] (0,-.5) -- (0,4);
\draw (0,4.2) node {$r_2$};
\draw (-0.4,-0.4) node {O};
\draw [dashed,->] (-0.2,-0.2) -- (3.5,3.5);
\begin{scope}[xshift=30pt, yshift=30pt]
\draw[line width= 0.4mm, color=red,scale=1,domain=0:1.5,smooth,variable=\t]
plot (\t,{-\t^1.5+2.5*\t});
\draw[-{Stealth[length=3mm, width=2mm]}, color=red](0.7,1.164) -- (1,1.5);
\draw (-0.4,0.1) node {\textcolor{red}{\tiny $\theta=0$}};
\draw (1.5,1.913+0.2) node {\textcolor{red}{\tiny $\theta=\pi/2$}};
 \end{scope}
\begin{scope}[xshift=80, yshift=80]
\draw[line width= 0.4mm, color=orange,scale=1,domain=0:1,smooth,variable=\t]
  plot (-\t,{\t^3-3*\t});
\draw[{Stealth[length=3mm, width=2mm]}-, color=orange](-0.7,-1.757) -- (-0.6,-1.584);
\draw (0.4,-0.1) node {\textcolor{orange}{\tiny $\theta=0$}};
\draw (-1,-2-0.2) node {\textcolor{orange}{\tiny $\theta=\pi/2$}};
\end{scope}
\end{tikzpicture}
\end{subfigure}
    \begin{subfigure}{0.45\textwidth}
    \centering
    \begin{tikzpicture}
\draw [->] (-.5,0) -- (4,0);
\draw (4.2,0) node {$r_1$};
\draw [->] (0,-.5) -- (0,4);
\draw (0,4.2) node {$r_2$};
\draw (-0.4,-0.4) node {O};
\draw [dashed,->] (-0.2,-0.2) -- (3.5,3.5);
\begin{scope}[xshift=20pt, yshift=20pt, scale=0.75]
\draw[line width= 0.4mm, color=red,scale=1,domain=0:3.141,smooth,variable=\t]
plot ({\t+0.5*sin(2*\t r)},{\t+0
1.5*sin(2*\t r)});
\draw[-{Stealth[length=3mm, width=2mm]}, color=red]({0.7+0.5*sin(2*0.7 r)},{0.7+
1.5*sin(2*0.7 r)}) -- ({0.71+0.5*sin(2*0.71 r)},{0.71+
1.5*sin(2*0.71 r)});

\draw[-{Stealth[length=3mm, width=2mm]}, color=red]({2.8+0.5*sin(2*2.8 r)},{2.8+
1.5*sin(2*2.8 r)}) -- ({2.81+0.5*sin(2*2.81 r)},{2.81+
1.5*sin(2*2.81 r)});
 \end{scope}
\end{tikzpicture}
\end{subfigure}
\caption{Left: RoC diagrams parameterised w.r.t. $\theta$. Right: $\mathfrak{R}(\mathcal{S})$ of a surface $\mathcal{S}$ satisfying $r_2-r_1=\sin 2\theta$. At $\theta=\pi/2$, $\mathcal{S}$ has a non-isolated umbilic hence $\mathfrak{R}(\mathcal{S})$ has a vertical tangent line when $\theta=\pi/2$.}
\label{fig: vert slope}
\end{figure}
We consider first the case when $p$ is an umbilic point of $\mathcal{S}$ which lies off the axis of rotational symmetry, in this case the Gauss angle of $p$, namely $\theta_0$, satisfies $\theta_0\neq 0,\pi$ and, due to the rotational symmetry of $\mathcal{S}$, $p$ lies in a curve consisting only of umbilic points, so is a non-isolated umbilic.
\begin{Thm}\label{thm: slope is unbounded}
Suppose $p\in\mathcal{S}$ is an umbilic which lies off the axis of symmetry of $\mathcal{S}$ and that $\mathcal{S}$ is non-flat at $p$. In addition assume that there is a punctured neighbourhood $U$ around $\theta_0$ such that $r_2(\theta)\neq r_1(\theta)$ on $U$. Then $\mu(\theta)$ is unbounded on $U$.
\end{Thm}
\begin{proof}
 We may assume $U\subset (0,\pi)$ since $p$ lies off the axis of symmetry. 
Let $s(\theta)=r_2(\theta)-r_1(\theta)$. By assumption there is a $\delta>0$ such that $s$ has no zeros on the interval $(\theta_0,\theta_0+\delta)\subset U$. We may also assume wlog that $\frac{\pi}{2}\not\in(\theta_0,\theta_0+\delta)$ so $\cot\theta$ has no zeros in this interval either. Hence by continuity both $s(\theta)$ and $\cot\theta$ take a fixed sign on $(\theta_0,\theta_0+\delta)$.
By taking $\theta_1=\theta_0$ and $\theta_2=\theta_0+\delta$ in the integrated Codazzi-Mainardi equation (\ref{eq: integrated cm}) we find that for $\theta\in(\theta_0,\theta_0+\delta)$
\[
|r_1(\theta)-r_0|=\int^\theta_{\theta_0}|s(\tau)|\cot\tau|\td \tau>0,
\]
hence $\mu(\theta)=\frac{r_2(\theta)-r_0}{r_1(\theta)-r_0}$ is well defined and may be written as
\begin{equation*}
\mu(\theta)=1+\frac{s(\theta)}{\int^\theta_{\theta_0}s(\tau)\cot\tau \td \tau},
\end{equation*}
which gives the following equation for $s(\theta)$:
\begin{equation}\label{eq: s in terms of slope theta_0}
    s(\theta)=(\mu(\theta)-1)\int^\theta_{\theta_0}s(\tau)\cot\tau \td \tau.
\end{equation}
Suppose for contradiction that $\mu(\theta)$ is bounded on $(\theta_0,\theta_0+\delta)$. Then
\[
|s(\theta)|\leq M\int^\theta_{\theta_0}|s(\tau)||\cot \tau|\td \tau,
\]
for some $M>0$. Performing the substitution $J(\theta)=\int^\theta_{\theta_0}|s(\tau)||\cot \tau|\td \tau$ and applying Gr\"onwall's inequality shows $s\equiv 0$ on $(\theta_0,\theta_0+\delta)$, which is a contradiction. Hence $\mu$ is unbounded.
\end{proof}
\begin{Cor}
 If $\mathfrak{R}(\mathcal{S})$ has a well defined tangent line at $(r_0,r_0)$, it must be vertical.
\end{Cor}
\noindent Now we consider when $p$ is isolated and must therefore lie on the axis of rotational symmetry. Without loss of generality we may assume $p$ is at the north pole of $\mathcal{S}$ so that $\theta_0=0$. In \cite{Galvez22} it was shown that when a surface has an isolated umbilic point, the slope of $\mathfrak{F}(\mathcal{S})$ must be positive as it meets the umbilic axis. Surfaces are necessarily convex around umbilic points - we show with the stronger assumption of strict convexity, a better bound can be given in the rotationally symmetric setting.

\begin{Lem}\label{lem: slope geq 1}
Let $\mathcal{S}$ be strictly convex at an  isolated umbilic point $p$ on the axis of rotational symmetry. Then $\liminf\limits_{\theta\to0}\mu(\theta) \geq 1$.
\end{Lem}
\begin{proof}
Since $p$ is isolated there exists an interval $(0,\delta)$, with $\delta<\frac{\pi}{2}$ on which $s(\theta)=r_2(\theta)-r_1(\theta)$ is non-zero. As before $\mu(\theta)$ can be given as
\begin{equation}\label{eq: mu in terms of s}
\mu(\theta)=1+\frac{s(\theta)}{\int^\theta_{0}s(\tau)\cot\tau \td \tau},
\end{equation}
which is well defined in $(0,\delta)$ since $p$ is isolated. For all $\theta\in(0,\delta)$, $s(\theta)$ is of a fixed sign and $\cot\theta>0$, thus $\int^\theta_0s(\tau)\cot\tau~\td \tau$ is of the same sign as $s(\theta)$. It follows that $\mu(\theta)\geq 1$ for all $\theta\in(0,\delta)$.
\end{proof}
\begin{Cor}
    Let $\mathcal{S}$ be strictly convex at an isolated umbilic point $p$ on the axis of rotational symmetry. When $\mu_p$ exists, $\mu_p\geq 1$.
\end{Cor}
\noindent We now go on to bound $\mu_p$ below by bounds larger than $1$, the bounds depending on the geometry of $\mathcal{S}$ near $p$. First a technical lemma is established.

\begin{Lem}\label{lem: tech lemma}
Let $\beta \geq 0$ and let $J:[0,c) \to \mathbb{R}$ be a continuous function on $[0,c)$ and differentiable on $(0,c)$ for some $c\geq 0$. Furthermore, suppose that $J$ is not identically $0$ on any punctured neighbourhood of $0$ and that
\[
\liminf\limits_{x\to 0}\left(\frac{\tan x J'(x)}{J(x)}\right)>-\infty.
\]
\begin{enumerate}[label=(\Alph*)]
    \item   If $\lim\limits_{x\to 0}\left(\frac{J(x)}{\sin^\beta x}\right)$ is finite then $\limsup\limits_{x\to 0}\left(\frac{\tan x J'(x)}{J(x)}\right)\geq \beta$.
     \item  If $\liminf\limits_{x\to 0}\left(\frac{\tan x J'(x)}{J(x)}\right)>\beta$, then  $\lim\limits_{x\to 0}\left(\frac{J(x)}{\sin^\beta x}\right)=0$.
\end{enumerate}
\end{Lem}
\begin{proof}
First we prove (A). Note
\[
\limsup\limits_{x\to 0}\left(\frac{\tan x J'(x)}{J(x)}\right)\geq\liminf\limits_{x\to 0}\left(\frac{\tan x J'(x)}{J(x)}\right)>-\infty,\] and if $\limsup\limits_{x\to 0}\left(\frac{\tan x J'(x)}{J(x)}\right)= \infty$ we are done. Hence the case $\limsup\limits_{x\to 0}\left(\frac{\tan x J'(x)}{J(x)}\right)=M$ for some $M\in\mathbb{R}$ is all that need be considered. Let $\varepsilon>0$. There exists some $\delta$ satisfying $0<\delta<c$ (which we may assume to be less than $\pi/2$) such that
\[
\frac{\tan x J'(x)}{J(x)}\leq M+\varepsilon, \qquad \text{ for all }x\in(0,\delta).
\]
If $0<z<y<\delta$, dividing through the above inequality by $\tan x$ and integrating between $x=z$ and $x=y$ gives the inequality $|J(y)|\leq |J(z)|\left(\frac{\sin y}{\sin z}\right)^{M+\varepsilon}$ implying
\begin{equation}\label{eq: upper bound on J(y)}
|J(y)|\leq \left|\frac{J(z)}{\sin^\beta z}\right|( \sin^{\beta-(M+\varepsilon)}z)(\sin^{M+\varepsilon}y), \qquad 0<z<y<\delta.
\end{equation}
If we assume for a contradiction that $M<\beta$, taking $\varepsilon$ such that $M+\varepsilon<\beta$ and letting $z\to 0$ inequality (\ref{eq: upper bound on J(y)}) implies implying $J(y)=0$ for all $y\in (0,\delta)$ which is a contradiction. Hence $M\geq\beta$.
\\ \hfill \\
Now (B) is proven. Let $N\in\mathbb{R}$ be such that $\liminf\limits_{x\to 0}\left(\frac{\tan x J'(x)}{J(x)}\right)\geq N>\beta$.
\\ \hfill \\
Let $\varepsilon>0$. There exists some $\delta$ satisfying $0<\delta<c$ (which we may again assume to be less than $\pi/2$) such that
\[
N-\varepsilon\leq \frac{\tan x J'(x)}{J(x)} \qquad \text{ for all }x\in(0,\delta).
\]
If $0<z<y<\delta$, dividing through the above inequality by $\tan x$ and integrating between $x=z$ and $x=y$ gives the inequality $|J(z)|\left(\frac{\sin y}{\sin z}\right)^{N-\varepsilon}\leq|J(y)|$, or after re-arrangement
\begin{equation}\label{eq: lower bound on J(y)}
\frac{|J(z)|}{\sin^\beta z}\leq \frac{|J(y)|}{\sin^{N-\varepsilon}y}\cdot \sin^{N-\varepsilon-\beta}z
\end{equation}
Taking $\varepsilon>0$ such that $N>\beta+\varepsilon$ and letting $z\to0$ in inequality (\ref{eq: lower bound on J(y)}) gives the claimed result.
\end{proof}
\begin{Thm}\label{thm: weak general slope restriction}
Let $\mathcal{S}$ be strictly convex at an isolated umbilic point $p$ on the axis of rotational symmetry.
\begin{enumerate}[label=(\Alph*)]
\item If the limit $\lim\limits_{ \theta \to 0}\left( \frac{r_2-r_1}{\sin^\alpha \theta}\right)$ is finite for some $\alpha\geq 0$, then $\limsup\limits_{\theta\to0}\mu(\theta)\geq \alpha +1$.
\item Conversely if $\liminf\limits_{\theta\to0}\mu(\theta)>\alpha+1$ then $\lim\limits_{ \theta \to 0}\left( \frac{r_2-r_1}{\sin^\alpha \theta}\right)=0$.
\end{enumerate}
\end{Thm}
\begin{proof}
 First consider the $\alpha=0$ case separately. Since $\mathcal{S}$ is assumed strictly convex
 \[
\limsup\limits_{\theta\to0}\mu(\theta)\geq \liminf\limits_{\theta\to0}\mu(\theta)\geq 1,
 \]
 with the last inequality holding by Lemma \ref{lem: slope geq 1}. Thus the conclusion of (A) holds. Also, again by strict convexity, $r_1(0)=r_2(0)=r_0<\infty$, where $r_0$ is the finite radii of curvature of $\mathcal{S}$ at $p$. Hence $\lim\limits_{ \theta \to 0}\left( \frac{r_2-r_1}{\sin^\alpha \theta}\right)=\lim\limits_{ \theta \to 0}\left(r_2-r_1\right)=0$. Hence the conclusion of (B) holds.
 \\ \hfill \\
 Now assume $\alpha>0$. In line with the notation of \Cref{lem: tech lemma}, let 
 \[
 J(\theta)=r_1(\theta)-r_0.
 \]
 It follows from equation (\ref{eq: mu in terms of s}) that
 \[
 \mu(\theta)=1+\frac{\tan\theta J'(\theta)}{J(\theta)}.
 \]
 Note that since $p$ is an isolated umbilic, $J(\theta)$ is non-zero in an interval $(0,c)$ for some $c>0$ and by the Codazzi-Mainardi equation (\ref{eq: differentiated Codazzi-Mainardi}) $J(\theta)$ is differentiable on $(0,c)$. Note we may take $c<\pi/2$ so that $\tan\theta>0$ and the Codazzi--Mainardi equation is not singular on $(0,c)$. Furthermore by Lemma \ref{lem: slope geq 1}
 \[
 \liminf\limits_{\theta\to 0}\left(\frac{\tan \theta J'(\theta)}{J(\theta)}\right)=\liminf\limits_{\theta\to 0}(\mu(\theta)-1)>0.
 \]
 Hence $J(\theta)$ satisfies the prerequisites of Lemma \ref{lem: tech lemma}. Statements (A) and (B) now follow from the respective statements in Lemma \ref{lem: tech lemma}.
 \end{proof}
In the special case we have the umbilic slope $\mu_p$ existing we have the following corollary
\begin{Thm}\label{thm: general slope restriction}
Let $\mathcal{S}$ be strictly convex at an isolated umbilic point $p$ on the axis of rotational symmetry. Suppose at $p$, $\mathcal{S}$ has an umbilic slope of $\mu_p\in\mathbb{R}$.
\begin{enumerate}[label=(\Alph*)]
\item If the radii of curvature satisfy
$\lim\limits_{ \theta \to 0}\left( \frac{r_2-r_1}{\sin^\alpha \theta}\right)=\gamma$ for some $\alpha\geq 0$, $\gamma\in\mathbb{R}$, then $\mu_p \geq \alpha+1$, with equality if $\gamma \neq 0$.
\item Conversely if $\mu_p\in\mathbb{R}$ satisfies $\mu_p>\alpha+1$ then $\lim\limits_{ \theta \to 0}\left( \frac{r_2-r_1}{\sin^\alpha \theta}\right)=0$.
\end{enumerate}
\begin{proof}
Since $\mu_p$ exists,
\[
\mu_p=\lim_{\theta\to0}\mu(\theta)=\limsup\limits_{\theta\to0}\mu(\theta)=\liminf\limits_{\theta\to0}\mu(\theta).
\]
Statement (A) almost follows by (A) of Theorem \ref{thm: weak general slope restriction}, we just need to show that $\mu_p=\alpha+1$ when $\gamma\neq 0$. If we assume this is the case then by L'H\^opitals rule
\begin{align*}
\mu_p&=1+\lim\limits_{\theta \to 0} \left(\frac{s(\theta)}{\int^\theta_{0}s(\tau)\cot\tau \td \tau}\right),\\
&=1+\lim\limits_{\theta \to 0}\left(\frac{s(\theta)}{\sin^\alpha \theta} \right) \cdot \lim\limits_{\theta \to 0}\left(\frac{\sin^\alpha \theta}{\int^\theta_0 s(\tau) \cot \tau \td \tau}\right),\\
&=\gamma \cdot \frac{\alpha}{\gamma}+1,\\
&=\alpha+1.
\end{align*}
Statement (B) follows directly from (B) of Theorem \ref{thm: weak general slope restriction}.
\end{proof}
\end{Thm}
We remark that $\mu_p=\alpha+1$ is not sufficient to determine the behaviour of $\frac{r_2-r_1}{\sin^\alpha \theta}$ as $\theta \to 0$. In particular, the strict inequality in \Cref{thm: general slope restriction} part (B) is tight, as the next example shows.
\begin{Ex}\label{ex: different limit behaviour}
Consider the family of surfaces of revolution satisfying
\[
r_2(\theta)-r_1(\theta)=\sin^\alpha \theta \cdot \ln(2\csc \theta)^\beta
\]
for $\beta\in \mathbb{R}$. The existence of these surfaces is shown by solving the above equation for $r_2-r_1$ together with the Codazzi-Mainardi equation (\ref{eq: differentiated Codazzi-Mainardi}) to find $r_1(\theta)$ and $r_2(\theta)$, and then applying \Cref{prop: curve and CM iff surface}. We remark that $\mu_p=\alpha+1$ for each of these surfaces but
\[
\lim\limits_{\theta\to 0}\left(\frac{r_2(\theta)-r_1(\theta)}{\sin^\alpha\theta}\right)=\begin{cases}
    0 & \beta=-1\\
    1 &\beta=0\\
    \infty &\beta=+1
\end{cases}.
\]
\end{Ex}
\begin{Cor}\label{cor: slope if c4 or 4}
    Let $p$ be an isolated umbilic point on the axis of rotational symmetry of $\mathcal{S}$, a strictly convex and $C^3$-smooth surface. Then if $\mu_p$ exists, $\mu_p\geq 2$. If in addition $\mathcal{S}$ is $C^4$-smooth then $\mu_p\geq 3$.
\end{Cor}
\begin{proof}
Since $\mathcal{S}$ is $C^3$-smooth, Lemma (\ref{lem: C3 surf gives C1 radii}) implies the radii of curvature are $C^1$-smooth and in particular have vanishing derivative at $\theta=0$. By L'H\^opital's rule
    \[
\lim\limits_{\theta \to 0}\left(\frac{r_2-r_1}{\sin\theta}\right)=\lim\limits_{\theta \to 0}\left(\frac{r_2'-r_1'}{\cos\theta}\right)=r_2'(0)-r_1'(0)=0,
    \]
and hence by \Cref{thm: general slope restriction} $\mu_p\geq 2$. If $\mathcal{S}$ is $C^4$-smooth then by a similar argument, applying Lemma (\ref{lem: C4 surf gives C2 radii});
  \[
\lim\limits_{\theta \to 0}\left(\frac{r_2-r_1}{\sin^2\theta}\right)=\frac{r_2''(0)-r_1''(0)}{2}.
    \]
Since this limit exists \Cref{thm: general slope restriction} implies $\mu_p\geq 3$.
\end{proof}

\vspace{0.1in}
\section{$\slr$ Transformations}
\label{sec: slr trans}
We first motivate the study of $\slr$ transformations of RoC space by discussing their geometrical significance. Consider the anti de-Sitter metric on RoC space, given in $(\psi,s)$ coordinates:
\begin{align}\label{eq: psi s coord}
   &g=\frac{\td\psi^2-\td s^2}{s^2}, &(\psi,s)=\left(\frac{r_2+r_1}{2},\frac{r_2-r_1}{2}\right).
\end{align}
The quantities $(\psi,s)$ are natural when discussing the set of oriented normal lines of a surface, which is a subset of the space of oriented lines in $\mathbb{R}^3$, denoted as $\mathbb{L}$. $\mathbb{L}$ is a 4 dimensional real manifold which is diffeomorphic to $T\mathbb{S}^2$ \cite{gk04} and carries a canonical K\"ahler structure with metric $\mathbb{G}$, of signature $(2,2)$ induced from the round metric on $\mathbb{S}^2$ \cite{gk05}. Given a surface in $\mathbb{R}^3$, its oriented normal lines form a surface in $\mathbb{L}$ which inherits a sub-manifold geometry from $\mathbb{G}$. The anti de-Sitter metric for RoC space arises from pushing forward this geometry to RoC space. See \cite{gk18}, Theorem 9, page 8 for further details.
There is a simple description of the isometries of $(\mathbb{R}^2,g)$ once we have complexified $\mathbb{R}^2$ with the split complex variable $j$ satisfying $j^2=+1$. 
\begin{Prop}[\cite{Kisil12}, Lemma 9.3, page 118]\label{prop: isom group action split plane}
The fractional linear transformations
\[
z \mapsto \frac{az+b}{cz+d}, \qquad \qquad \begin{pmatrix}a & b \\ c & d \end{pmatrix}\in \slr,
\]
where $z=\psi+js$ are isometries of $(\mathbb{R}^2,g)$. 
\end{Prop}
From this point on-wards, the above transformations are referred to as $\slr$ transformations and the symbols $a,b,c,d \in \mathbb{R}$ will exclusively denote the entries of a general element of $\slr$ as done in \Cref{prop: isom group action split plane}. The geodesics of this metric correspond to the RoC diagrams of surfaces which are linear in their Gauss and mean curvature, i.e. LW-surfaces.
\begin{Prop}
The geodesics of $(g,\mathbb{R}^2)$ are the curves satisfying
\begin{equation}\label{eq: H, K linear relation}
\alpha H+\beta K=\gamma
\end{equation}
where $\alpha,\beta,\gamma \in \mathbb{R}$ and $H$ and $K$ are the mean and Gauss curvature respectively.
\end{Prop}
\begin{proof}
The 2D anti de-Sitter space is maximally symmetric, hence there are three linearly independent Killing vectors of $g$ denoted $U_i$, $i=1,2,3$ which we give in $(\psi,s)$ coordinates;
\begin{align*}
    &U_1=(\psi^2+s^2)\partial_\psi+2\psi s \partial_s,& U_2= \psi \partial_\psi+s\partial_s,&& &U_3=\partial_\psi.&
\end{align*}

If $V$ is a geodesic tangent vector, the quantities 
\[
\lambda_i=g(U_i,V), \qquad i=1,2,3,
\]
are constant along geodesics and therefore may be used to algebraically describe them. A short calculation shows that
\begin{align*}
    &\lambda_1=\dfrac{\psi^2+s^2}{s^2}\frac{\td\psi}{\td \theta}-\dfrac{2\psi}{s}\frac{\td s}{\td \theta}& \lambda_2=\dfrac{\psi}{s^2}\frac{\td\psi}{\td \theta}-\dfrac{1}{s}\frac{\td s}{\td \theta}& &\lambda_3=\dfrac{1}{s^2}\frac{\td \psi}{\td \theta}&,
\end{align*}
where $V=\left(\dv{\psi}{\theta},\dv{s}{\theta}\right)$. Eliminating $V$ from these equations gives an algebraic equation defining the geodesics of $g$:
\begin{align}\label{eq: ADS geodesic}
    &\left(\psi-\dfrac{\lambda_
    2}{\lambda_3}\right)^2-s^2=\dfrac{\lambda_2^2-\lambda_1\lambda_3}{\lambda_3^2}, & & \lambda_3 \neq 0,\\
    &\psi=\textstyle{ constant},& &\lambda_3=0 .
\end{align}
Writing the above equations in terms of the principal curvatures $k_1$ and $k_2$ via equation (\ref{eq: psi s coord}) gives the stated relationship (\ref{eq: H, K linear relation}) for some constants $\alpha$, $\beta$, $\gamma \in \mathbb{R}$.
\end{proof}
The action of the $\slr$ transformations on RoC space can be given in $(r_1,r_2)$ coordinates, denote the corresponding coordinate transformation as $T_M:\mathbb{R}^2\to\mathbb{R}^2$, for a given $M\in \slr$.
\begin{Prop}\label{prop: coordinate transform}
 Under a general $\slr$ transformation the radii of curvature are mapped to
\[
T_M(r_1,r_2)= \left(\frac{ar_1+b}{cr_1+d},\frac{ar_2+b}{cr_2+d}\right),
\]
$M$ a general element of $\slr$.
\end{Prop}
\begin{proof}
Denote the image of $z$ under the fractional linear transformation given in Proposition \ref{prop: isom group action split plane} as $\widetilde{z}$. Denote the real and imaginary parts of $\widetilde{z}$ as $\widetilde{\psi}$ and $\widetilde{s}$ respectively, in accordance with equation (\ref{eq: psi s coord}), $\widetilde{r}_1$ and $\widetilde{r}_2$ are then given by
\begin{align*}
&\widetilde{r}_1=\widetilde{\psi}+\widetilde{s}, & \widetilde{r}_2=\widetilde{\psi}-\widetilde{s}.
\end{align*}
The rest of the proof is now an exercise in writing $\widetilde{r_i}$ in terms of $r_i$, $i=1,2$. In terms of $\psi$ and $s$, we have
\begin{align*}
&\widetilde{\psi} = \mathrm{Re}(\widetilde{z})= \mathrm{Re}\left(\frac{az+b}{cz+d}\right)=\frac{ac(\psi^2-s^2) +(ad+bc)\psi+bd}{(c\psi+d)^2-c^2s^2},
\end{align*}
and
\begin{align*}
&\widetilde{s}=\mathrm{Im}(\widetilde{z})=\mathrm{Im}\left(\frac{az+b}{cz+d}\right)=\frac{s}{(c\psi+d)^2-c^2s^2}.
\end{align*}
Solving for $(\widetilde{r}_1,\widetilde{r}_2)$ in terms of $(r_1,r_2)$ using equation (\ref{eq: psi s coord}) finishes the proof.
\end{proof}
\begin{remark}
The map $T_M$ is extended to all of RoC space, i.e. to $\widehat{\mathbb{R}} \cross \widehat{\mathbb{R}}$ by extending the domain and range of each of its component maps from $\mathbb{R}$ to $\widehat{\mathbb{R}}$ as is often done with fractional linear transformations.
\end{remark}
\begin{remark}\label{rem: prin curv transform}
We will abbreviate by $\widetilde{r}_i$ the image of $r_i$ under a $\slr$ transformation as in Proposition \ref{prop: coordinate transform} when the transformation in question is clear. We let $\widetilde{k_i}=1/\widetilde{r}_i$ for i=1,2, which are the transformed principal curvatures. It is easy to show that
\begin{equation}\label{eq: curv transform in terms of k}
\widetilde{k_i}=\frac{dk_i+c}{bk_i+a}.
\end{equation}
which is the corresponding transformation on curvature space $\mathfrak{F}(\mathcal{S})$ stated in the introduction.
\end{remark}

\vspace{0.1in}
\subsection{Induced Surface Transformations}

The Codazzi-Mainardi equation (\ref{eq: differentiated Codazzi-Mainardi}) will now be used as an integrability condition to find a (possibly non-regular) surface $\widetilde{\mathcal{S}}$ such that $\mathfrak{R}(\widetilde{\mathcal{S}})$ is the image, under a $\slr$ transformation, of an initial RoC diagram $\mathfrak{R}(\mathcal{S})$ for some $\mathcal{S}$.  In what follows let $a,b,c,d$ represent the coefficients of a general $M\in\slr$.
\begin{Lem}\label{prop: param of image under sl2r}
If $\mathcal{J}:I \to \mathbb{R}^2$ is parameterised by $\theta\in I\subseteq[0,\pi]$ and satisfies the Codazzi-Mainardi equation (\ref{eq: differentiated Codazzi-Mainardi}), then if $M\in\slr$, and $r_i \not\equiv -d/c$ for $i=1,2$, $T_M(\mathcal{J}(I))$ can be parameterised to satisfy the Codazzi-Mainardi equation. In particular, if $\widetilde{\theta}$ satisfies
\begin{equation}\label{eq: theta reparam}
\sin\widetilde{\theta}=\mathcal{A}\sin\theta \cdot (cr_1(\theta)+d), \enskip \mathcal{A}\in\mathbb{R}\backslash\{0\},
\end{equation}
then $\widetilde{\theta}$ is such a parametrisation of $T_M(\mathcal{J}(I))$ .

\end{Lem}
\begin{proof}
We have that
\begin{equation}
\label{eq: T_M(J)(I))}
T_M(\mathcal{J}(I))=\left\{(\widetilde{r}_1,\widetilde{r}_2)\in\widehat{\mathbb{R}}\times\widehat{\mathbb{R}}\left| \enskip \widetilde{r}_i =\frac{ar_i+b}{cr_i+d}, \enskip i=1,2, \enskip (r_1,r_2)\in\mathcal{J}(I)\right.\right\}.
\end{equation}
Assume that $\widetilde{\theta}$ parameterises a part of $T_M(\mathcal{J}(I))$ for which $cr_i+d\neq0$. Suppose $\widetilde{\theta}$ satisfies equation (\ref{eq: theta reparam}). Then
\begin{align*}
    \dv{\widetilde{{r}_1}}{\widetilde{\theta}}&= \dv{\theta}{\widetilde{\theta}}\cdot \dv{\widetilde{{r}_1}}{\theta} =\frac{\cos\widetilde{\theta}}{\mathcal{A}\cos\theta\cdot (cr_2+d)}\dv{}{\theta} \left(\frac{ar_1+b}{cr_1+d}\right),
\end{align*}
where we have implicitly differentiated equation (\ref{eq: theta reparam}) by $\theta$ to find $\rm{d}\theta/\rm{d}\widetilde{\theta}$, and used the Codazzi-Mainardi equation to remove any derivatives of $r_1$. A short calculation, removing any derivatives of $r_1$ by Codazzi-Mainardi, and removing occurrences of $\sin \theta$ by equation (\ref{eq: theta reparam}) yields the result
\begin{align*}
    \dv{\widetilde{{r}_1}}{\widetilde{\theta}}=(\widetilde{r}_2-\widetilde{r}_1)\cot\widetilde{\theta}.
\end{align*}
\end{proof}
\begin{remark}
    In fact, away from points of $\mathcal{J}(I)$ for which $r_1=r_2$, parameters $\widetilde{\theta}$ satisfying relation (\ref{eq: theta reparam}) are the only parameters for which $T_M(\mathcal{J}(I))$ satisfies the Codazzi-Mainardi equation. Indeed if
\begin{align*}
    & \dv{{r}_1}{\theta}=(r_2-r_1)\cot\theta, & \text{and} && \dv{\widetilde{{r}_1}}{\widetilde{\theta}}=(\widetilde{r}_2-\widetilde{r}_1)\cot\widetilde{\theta},
\end{align*}
hold on $\mathcal{J}(I)$ and $T_M(\mathcal{J}(I))$ respectively, then the quotient of these equations is a separable ODE which solves to give $\widetilde{\theta}$ implicitly by equation (\ref{eq: theta reparam}).
\end{remark}

We now prove the main theorem of Section \ref{sec: slr trans}, stated in terms of $\mathfrak{R}(\mathcal{S})$ instead of $\mathfrak{F}(\mathcal{S})$ as was done in the introduction.
\begin{Thm}\label{thm: new surface form old with sl2r}
If $T_M$ is a real fractional linear transformation and $\mathcal{S}$ is non-flat, then there exists a rotationally symmetric and possibly non-regular surface $\widetilde{\mathcal{S}}$ such that $\mathfrak{R}(\widetilde{\mathcal{S}})=T(\mathfrak{R}(\mathcal{S}))$.
\end{Thm}
\begin{proof}
Using the notation of \Cref{prop: param of image under sl2r} we first consider the case where $r_i \equiv -d/c$ for $i=1$ or $2$. If $r_1\equiv -d/c$ then the Codazzi-Mainardi equation implies $r_2\equiv -d/c$ in which case $\mathcal{S}$ is a round sphere. Hence $T_M(\mathfrak{R}(\mathcal{S}))=\{(\infty,\infty)\}$ and $\widetilde{\mathcal{S}}$ is a plane. On the other hand if $r_2\equiv -d/c$ and $r_1\not \equiv -d/c$ then $T_M(\mathfrak{R}(\mathcal{S}))=\{(\widetilde{r_1},\infty)\}$ which is the RoC diagram of a cone. Now assume that $r_i \not\equiv -d/c$ for $i=1,2$. If we set $\mathfrak{R}(\mathcal{S})=\mathcal{J}(I)$, the curve $T_M(\mathcal{J}(I))$ has a parametrisation satisfying Codazzi-Mainardi, and so by Proposition \ref{prop: curve and CM iff surface} corresponds to the RoC diagram of a $C^2$-smooth surface of revolution with the claimed parametrisation. For each value of $\mathcal{A}\neq 0$, we get a different surface with $\mathfrak{R}(\widetilde{\mathcal{S}}_\mathcal{A})=T_M(\mathcal{J}(I))$. We also remark that these surfaces may be non-regular since if $\mathcal{S}$ has a point for which $r_i=-b/a$, then $\widetilde{r_i}=0$ and $\widetilde{\mathcal{S}}$ will have a cusp.
\end{proof}
The surfaces generated by integrating the image of a $\slr$ transformation can be understood in $\mathbb{E}^3$ by the following relations
\begin{Prop}\label{prop: transformations on theta and rho}
Let $\widetilde{\mathcal{S}}$ be the image of $\mathcal{S}$, non-flat, as described in Theorem \ref{thm: new surface form old with sl2r}. Let $\rho(\theta)$ and $\widetilde{\rho}(\widetilde{\theta})$  be the distances from the axis of rotation of $\mathcal{S}$ and $\widetilde{\mathcal{S}}$ respectively. Then
\begin{align}\label{eq: transformation equation}
&\sin\widetilde{\theta}=\mathcal{A}\left(c\rho(\theta)+d\sin\theta\right), &\widetilde{\rho}(\widetilde{\theta})=\mathcal{A}\left(a\rho(\theta)+b\sin\theta\right).
\end{align}
\end{Prop}
\begin{proof}
The first equation follows directly from \Cref{prop: param of image under sl2r} and equation (\ref{eq: theta reparam}) by using the relation $\rho=r_1\sin\theta$. The second follows from
\begin{align*}
    \widetilde{\rho}(\widetilde{\theta})&=\widetilde{
    r}_1\sin\widetilde{\theta},\\
    &=\left(\frac{ar_1+b}{cr_1+d}\right) \cdot \mathcal{A}\sin\theta \cdot (cr_1+d),\\
    &=\mathcal{A}\left(a\rho(\theta)+b\sin\theta\right).
\end{align*}
\end{proof}
The constant $\mathcal{A}$ controls the speed with which $T_M(\mathcal{J}(I))$ is parameterised with respect to $\widetilde{\theta}$. Hence a fixed $\slr$ transformation induces a transformation in $\mathbb{E}^3$ of surfaces for each $\mathcal{A}\backslash\{0\}$ giving rise to different surfaces with a different range of Gauss angles $\widetilde{\theta} \in \tilde I \subset [0,\pi]$. In fact for a general $\mathcal{S}$, the map $\theta \mapsto \widetilde{\theta}(\theta)$ given by equation (\ref{eq: transformation equation}) will not always be well defined since the RHS may not lie between $1$ and $-1$ for all $\theta \in I$. If $\rho$ is bounded however, setting \[
\mathcal{A}= \left(\max\limits_{\mathcal{S}}{|c\rho+d\sin\theta|}\right)^{-1},
\]
ensures that the surface transformation can be defined for all $\theta \in I$. We now decompose the $\slr$ transformations into a composition of simpler transformations which have a clearer geometric interpretation. Consider the following subgroups of $\slr$
\begin{align*}
    &N=\left\{\begin{pmatrix}1 & v \\ 0 & 1 \end{pmatrix}: v\in \mathbb{R}\right\} &\text{and}
&&A=\left\{\begin{pmatrix}\omega & 0 \\ 0 & \frac{1}{\omega} \end{pmatrix}: \omega \neq 0\right\}.
\end{align*}
For these subgroups there is a natural choice of $\mathcal{A}$ which allows $\mathcal{S}$ and $\widetilde{\mathcal{S}}$ to be parameterised by the same Gauss angle, furthermore with this choice of $\mathcal{A}$ the induced transformations on surfaces are very geometric.
\begin{Prop}\label{prop: N is par translations}
The induced transformations of a surface in $\mathbb{E}^3$ by the subgroup $N$ with $\mathcal{A}=1$ are parallel translations.
\end{Prop}
\begin{proof}
Let the transformation $T_M$, $M=\begin{pmatrix}1 & v \\ 0 & 1 \end{pmatrix}\in N$ for $v\in \mathbb{R}$ act on RoC space. Then we have by Proposition \ref{prop: transformations on theta and rho}
\begin{align*}
&\sin\widetilde{\theta}=\mathcal{A}\sin\theta, &\widetilde{\rho}(\widetilde{\theta})=\mathcal{A}\left(\rho(\theta)+v\sin\theta\right).
\end{align*}
Therefore taking $\mathcal{A}=1$, we can take $\widetilde{\theta}=\theta$ and $\widetilde{\rho}=\rho+v\sin\theta$, furthermore by equations (\ref{eq: rho, h in terms of r}) we find $\widetilde{h}=h+v\cos\theta$. Since $\sin\theta$ and $\cos\theta$ are the radial and axial components of the unit normal vector of $\mathcal{S}$ at a point with Gauss angle $\theta$, the transformation translates every point of $\mathcal{S}$ a distance $v$ in the normal direction.
\end{proof}
\begin{Prop}\label{prop: A is homothety}
The action on a surface in $\mathbb{E}^3$ of the subgroup $A$ with $\mathcal{A}=\omega$ is homothety.
\end{Prop}
\begin{proof}
Proceeding as in the proof of Proposition \ref{prop: N is par translations}, let $M=\begin{pmatrix}\omega & 0 \\ 0 & \frac{1}{\omega} \end{pmatrix}\in A$, $\omega \neq 0$. Then $\widetilde{\theta}$ and $\widetilde{\rho}$ satisfy
\begin{align*}
&\sin\widetilde{\theta}=\frac{\mathcal{A}}{\omega}\sin\theta, &\widetilde{\rho}(\widetilde{\theta})=\omega\mathcal{A}\rho(\theta).
\end{align*}
Taking $\mathcal{A}=\omega$, gives $\widetilde{\theta}=\theta$, $\widetilde{\rho}=\omega^2\rho$ and again by equations (\ref{eq: rho, h in terms of r}) $\widetilde{h}=\omega^2h$. Hence the transformation is a homothety with a scale factor of $\omega^2$.
\end{proof}
Elements of $\slr$ have the following decomposition, if $c=0$ the general element can be written as
\begin{equation}
\begin{pmatrix} a & b \\ 0 &\td \end{pmatrix}=\begin{pmatrix}1 & ba \\ 0 & 1 \end{pmatrix}\begin{pmatrix}a & 0 \\ 0 & 1/a \end{pmatrix},
 \end{equation}
on the other hand if $c\neq 0$ then
 \begin{equation}
 \begin{pmatrix} a & b \\ c &\td \end{pmatrix}=\begin{pmatrix}1 & a/c \\ 0 & 1 \end{pmatrix}\begin{pmatrix}1/c & 0 \\ 0 & c \end{pmatrix}\begin{pmatrix}0 & -1 \\ 1 & 0 \end{pmatrix} \begin{pmatrix}1 & d/c \\ 0 & 1 \end{pmatrix}.
 \end{equation}
Hence any element of $\slr$ can be constructed from composition of elements from $A$, $N$ and the matrix
\begin{equation}
    Q=\begin{pmatrix}0 & -1 \\ 1 & 0 \end{pmatrix},
\end{equation}
\noindent We now describe the action of $Q$ in more detail. As a corollary of Proposition \ref{prop: transformations on theta and rho} we have

\begin{Prop}\label{prop: reciprocal trm}
The matrix $Q$ induces a 1-parameter family of transformations mapping $\mathcal{S}$ to a surface of revolution $\widetilde{\mathcal{S}}_\mathcal{A}$ satisfying
\begin{align}\label{eq: Q transformation theta and rho}
& \sin \widetilde{\theta}=\mathcal{A}\rho(\theta),
&\widetilde{\rho}(\widetilde{\theta})=-\mathcal{A} \sin \theta.
\end{align}
\end{Prop}
Unlike with elements in the subgroups $A$ and $N$ there is no obvious canonical choice of $\mathcal{A}$ to associate with $Q$ in the case of a general $\mathcal{S}$. If $\mathcal{S}$ is a $C^2$-smooth, closed, strictly convex and regular surface, we have the following.
\begin{Prop}
 If $\mathcal{S}$ is a $C^2$-smooth, closed, strictly convex and regular surface, then with $\mathcal{A}=\rho\left(\pi/2\right)^{-1}$, the transformation induced by $Q$ sends $\mathcal{S}$ to a closed, strictly convex and regular surface, $\widetilde{\mathcal{S}}$.
\end{Prop}
\begin{proof}
Take $\theta\in [0,\pi]$. Since $\mathcal{S}$ is strictly convex and regular, $r_1$ and $r_2$ are bounded, have the same sign and are non-zero, hence by continuity do not change sign on $\mathcal{S}$. Furthermore from relations (\ref{eq: rho in terms of r1}) and (\ref{eq: rho and h slope is -tan}) we see
\begin{equation}\label{eq: drho/dtheta in terms of r2}
\dv{\rho}{\theta}=r_2\cos\theta,
\end{equation}
implying $\theta=\pi/2$ is a stationary point of $\rho$. Furthermore by taking a difference quotient of the above equation we see that $\rho''(\pi/2)=-r_2(\pi/2)\neq0$ hence $\rho$ attains a local maxima (or minima) at $\theta=\pi/2$, depending on if $r_2(\theta)$ is positive (or negative) on $\mathcal{S}$. However $\rho$ takes the same sign as $r_1$ when $\theta\in[0,\pi]$ (via equation \ref{eq: rho in terms of r1}) which takes the same sign as $r_2$, thus $\theta=\pi/2$ is a maxima when $\rho(\theta)>0$ on $[0,\pi]$ and a minima when $\rho<0$ on $[0,\pi]$. Hence $|\rho|$ attains a maximum of $|\rho(\pi/2)|$ implying $-1\leq\frac{\rho(\theta)}{\rho(\pi/2)}\leq 1$. Letting $\mathcal{A}=\rho(\pi/2)^{-1}$ we can define the bijective and continuous map $\theta \mapsto \widetilde{\theta}$ satisfying equation (\ref{eq: Q transformation theta and rho}):
\begin{align}\label{eq: bijective re-param}
&\widetilde{\theta}(\theta)=
\begin{cases}
\sin^{-1}\left(\frac{\rho(\theta)}{\rho(\pi/2)}\right), &0\leq \theta \leq \frac{\pi}{2} \\ \pi-\sin^{-1}\left(\frac{\rho(\theta)}{\rho(\pi/2)}\right), &\frac{\pi}{2} \leq \theta \leq \pi 
\end{cases}.
\end{align}
By Proposition \ref{thm: new surface form old with sl2r}, $\widetilde{\theta}$ is the Gauss angle of a surface of revolution $\widetilde{\mathcal{S}}$. Note that since for $i=1,2$, $r_i$ is bounded, $\widetilde{r_i}$ is not zero, hence $\widetilde{\mathcal{S}}$ is regular. We now show that the profile curve of $\widetilde{\mathcal{S}}$ is the continuous image of the compact set $[0,\pi]$. It then follows that $\widetilde{\mathcal{S}}$ is compact.  Note that from relation (\ref{eq: Q transformation theta and rho}), $\tilde{\rho}$ is a bounded and continuous function of $\theta$ and from relations (\ref{eq: rho in terms of r1}),(\ref{eq: Q transformation theta and rho}) and the Codazzi-Mainardi equation, up to the addition of some constant of integration,
\begin{align}\label{eq: h tilde under reciprocal transform}
\tilde h &=-\int \tilde r_2 \sin \widetilde{\theta} \td \widetilde{\theta}=\int\frac{1}{\rho(\pi/2)}\frac{\rho(\theta)}{r_2(\theta)}\dv{\widetilde{\theta}}{\theta}\td \theta=\int\frac{\sgn(\pi/2-\theta)\rho(\theta) \cos \theta }{|\rho(\pi/2)|\sqrt{\rho(\pi/2)^2-\rho(\theta)^2}}\td \theta,
\end{align}
where $\td \widetilde{\theta}/ \td \theta$ is calculated from equation (\ref{eq: bijective re-param}). Hence to show $\tilde h$ is continuous on $[0,\pi]$, it is shown that the integrand in equation (\ref{eq: h tilde under reciprocal transform}) is continuous and bounded, in particular at $\theta = \pi/2$. Since $\rho$ is $C^2$ in the variable $\theta$, equation (\ref{eq: drho/dtheta in terms of r2}) Taylor's theorem with remainder $\omega$ gives;
\[
\rho(\theta)=\rho\left(\frac{\pi}{2}\right)+ \frac{1}{2}\rho''\left(\frac{\pi}{2}\right)\left(\theta-\frac{\pi}{2}\right)^2+\omega(\theta)\left(\theta-\frac{\pi}{2}\right)^2,
\]
where $\omega(\theta)\to 0$ as $\theta \to \pi/2$. This implies the asymptotic behaviour of the integrand 
\begin{align*}
\frac{\sgn(\pi/2-\theta)\rho(\theta) \cos \theta }{|\rho(\pi/2)|\sqrt{\rho(\pi/2)^2-\rho(\theta)^2}}\to \sgn(\rho(\pi/2))\sqrt{K(\pi/2)},
\end{align*}
as $\theta \to \pi/2$, where $K$ is the Gauss curvature of $\mathcal{S}$. Hence the profile curve of $\widetilde{\mathcal{S}}$ is the continuous image of the set $[0,\pi]$. Since $\tilde{\rho}\to 0$ as $\theta \to 0,\pi$, we have that $\widetilde{\mathcal{S}}$ is without boundary and therefore closed.
\end{proof}
The surface transformations induced by $Q$ are called \textit{reciprocal transformations} as their action on RoC space is $(r_1,r_2)\mapsto (-1/r_1,-1/r_2)$. The action of the reciprocal transformations on a surface in $\mathbb{E}^3$ may be understood as exchanging the functions $\sin\theta$ and $\rho(\theta)$, up to a scalar multiple, as illustrated by Figure \ref{fig: illust of recip maps}.
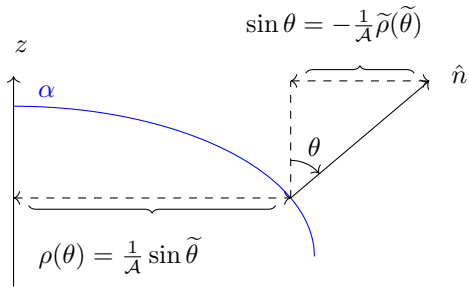
\begin{figure}[!h]
\centering
\begin{tikzpicture}[scale=2]
\draw [->] (0,-0.2) -- (0,1.2);
\draw (0,1.2)++(0.05,0.2) node {$z$};
\draw[color=blue,domain=0:3.141/2,smooth,variable=\x]
  plot ({2*sin(\x r)},{cos(\x r)});
\draw (0.21932,1.1) node {$\color{blue}\alpha$};
\coordinate (P) at (1.8421,0.3894);
\draw[->] (P) -- ++(.9211,0.7788) coordinate (n);
\draw (n) ++(0.2,0.05) node {$\hat{n}$};
\draw[dashed,-] (P) -- (n -| P) coordinate (N);
\pic[draw, <-, "$\theta$", angle eccentricity=1.5] {angle = n--P--N};
\draw[<->,dashed] (N)--(n);
\draw[decoration={brace,amplitude=1mm,mirror,raise=1mm},decorate] ([xshift = -0.1cm]n) -- node [above=4mm] {\parbox[t]{0.2\textwidth}{
                   $\sin \theta=-\frac{1}{\mathcal{A}}\widetilde{\rho}(\widetilde{\theta})$}} ([xshift = 0.1cm]N);
\draw[<->,dashed] (P -| {{(0,0)}}) -- (P);
\draw[decoration={brace,amplitude=1mm,mirror,raise=1mm},decorate] (P -| {{(0.1,0)}}) -- node [below=4mm] {\parbox[t]{0.2\textwidth}{
                   $\rho(\theta)=\frac{1}{\mathcal{A}}\sin \widetilde{\theta}$}} ([xshift = -0.1cm]P);
\end{tikzpicture}
\caption{A part of a profile curve $\alpha$ of a closed convex surface. The values of $\rho$ and $\theta$ determine the values of $\widetilde{\rho}$ and $\widetilde{\theta}$ of $\alpha$'s image under a reciprocal transformation as depicted.}\label{fig: illust of recip maps}
\end{figure}
The following theorem ties together the results of this section.
\begin{Thm}\label{thm: slr decomposition}
Any surface transformation induced by an element of $\slr$
is a composition of parallel translations, homotheties and reciprocal mappings.
\end{Thm}

\vspace{0.1in}
\subsection{General Properties of the $\slr$ Transformations.} \label{subsec: weingarten surface transformations}

The $\slr$ transformations generate examples of surfaces satisfying transformed Weingarten relations. If $\mathcal{S}$ has principal curvatures $(k_1,k_2)$, then its image $\widetilde{\mathcal{S}}$, has principal curvatures $(\widetilde{k}_1,\widetilde{k}_2)$ given by Remark \ref{rem: prin curv transform}. If $\eta$ denotes the map $\eta:(k_1,k_2)\mapsto (\widetilde{k}_1,\widetilde{k}_2)$ and $\mathcal{S}$ satisfies a Weingarten relation $W(k_1,k_2)=0$ we have
\begin{equation}\label{eq: push fwd}
(W\circ \eta^{-1})(\widetilde{k}_1,\widetilde{k}_2)=W(k_1,k_2)=0,
\end{equation}
hence dropping the tildes on the principal curvatures,  $\widetilde{\mathcal{S}}$ satisfies the Weingarten relationship $\widetilde{W}(k_1,k_2)=0$, where $\widetilde{W}=W\circ \eta^{-1}$. Furthermore, it allows us to relate surfaces satisfying different Weingarten relations through induced transformations in $\mathbb{E}^3$. Now we give some general properties of the $\slr$ transformations in terms of how they transform surfaces and their Weingarten relations.
\begin{Thm}\label{thm: properties of the slr}
    Let $\widetilde{\mathcal{S}}$ be the image of a surface of revolution $\mathcal{S}$ under the transformation of surfaces induced by $M \in \slr$.
\begin{enumerate}
\itemsep=-1em
    \item Umbilic points of $\mathcal{S}$ are mapped to umbilic points of $\widetilde{\mathcal{S}}$.\\
    \item When $\mathcal{S}$ is Weingarten, elliptic points of $\mathcal{S}$ are mapped to elliptic points of $\widetilde{\mathcal{S}}$, providing the principal curvatures at this point, $k_i$, satisfy $k_i\neq -a/b$.\\
    \item If $p\in\mathcal{S}$ and $\widetilde{p}\in\widetilde{\mathcal{S}}$ are both non-flat and isolated umbilic points with $\widetilde{p}$ being the image of $p$ under an $\slr$ transformation then  $\mu_{\widetilde{p}}=\mu_p$.
\end{enumerate}
\end{Thm}
\begin{proof}
These claims follow from the curvature transformations given in Proposition \ref{prop: coordinate transform}/Remark \ref{rem: prin curv transform}. Let $\mathcal{S}$ and $\widetilde{\mathcal{S}}$ have principal curvatures $(k_1,k_2)$ and $(\widetilde{k}_1,\widetilde{k}_2)$ respectively. The map $\eta:(k_1,k_2)\mapsto (\widetilde{k}_1,\widetilde{k}_2)$ as in equation (\ref{eq: push fwd}) is given explicitly as
\[
\eta(x,y)=\left(\frac{dx+c}{bx+a},\frac{dy+c}{by+a}\right).
\]
As long as $k_i\neq -a/b$,  $k_1=k_2$ if and only if $\widetilde{k}_1=\widetilde{k}_2$. If  Claim 2 follows by checking that
\begin{equation}\label{eq: derivative of transformed w}
{\pdv{\widetilde{W}(\widetilde{k}_1,\widetilde{k}_2)}{\widetilde{k}_i}}=\left(bk_i+a\right)^{2}\cdot{\pdv{W(k_1,k_2)}{k_i}},
\end{equation}
for $i=1,2$. Hence, recalling that for a Weingarten surface, ellipticity is equivalent to condition (\ref{eq: ellipticity cond}), it follows that $\mathcal{S}$ is elliptic at the point $q$ if and only if $\widetilde{S}$ is elliptic at the point $\widetilde{q}$, the image of $q$ under an $\slr$ transformation.
To show claim 3, first note that as $\widetilde{p}$ is a non-flat point, $\widetilde{r}_0\neq \infty$ and so $r_0\neq -d/c$. Note that after applying an $\slr$ transformation to the curvatures
\begin{align*}
\frac{\widetilde{r}_2-\widetilde{r}_0}{\widetilde{r}_1-\widetilde{r}_0}=\frac{\frac{ar_2+b}{cr_2+d}-\frac{ar_0+b}{cr_0+d}}{\frac{ar_1+b}{cr_1+d}-\frac{ar_0+b}{cr_0+d}}=\left(\frac{cr_1+d}{cr_2+d}\right)\left(\frac{r_2-r_0}{r_1-r_0}\right).
\end{align*}
Hence taking the limit of the above as $(r_1,r_2)\to(r_0,r_0)$ gives $\mu_{\widetilde{p}}=\mu_p$.
\end{proof}
\vspace{0.1in}
\subsection{Application to Semi-Quadratic Weingarten Surfaces}\label{subsec: app to qw surfaces}
A \textit{quadratic Weingarten surface} is any $C^2$-smooth surface satisfying the Weingarten relationship
\begin{equation}\label{eq: full quadratic Weingarten relationship}
    \tau k_2^2+\nu k_1^2+\alpha k_1k_2+\beta k_1 + \gamma k_2 +\delta =0, \qquad (\tau,\nu,\alpha,\beta,\gamma,\delta) \in \mathbb{R}^6\backslash\{0\},
\end{equation}
where $k_i$ are the principal curvatures. Special cases of this Weingarten relationship have been studied previously in the rotationally symmetric setting \cite{ks05}. We study the following subfamily.
\begin{Def}
The subfamily of quadratic Weingarten surfaces satisfying
\begin{equation}\label{eq: quadratic Weingarten relationship}
    \alpha k_1k_2+\beta k_1 + \gamma k_2 +\delta =0 \qquad \alpha,\beta,\gamma,\delta \in \mathbb{R},
\end{equation}
we call \textit{semi-quadratic Weingarten surfaces}.
\end{Def}
 This subfamily contains well studied classes of surfaces:
 \begin{align*}
&r_2=\lambda r_1+C, &k_2=\lambda k_1 + C, &&\lambda H+&\Upsilon K+C =0,&\\
&(\alpha\neq0,\delta=0) & (\alpha=0,\delta\neq0) &&(&\beta=\gamma)&
\end{align*}
for $\lambda,\Upsilon,C \in \mathbb{R}$, which are from left to right the linear Hopf surfaces which satisfy relation (\ref{eq: lin hopf surf}), $k$-\textit{linear surfaces} and \textit{LW-surfaces}, studied in \cite{gk21,lp20}, \cite{lopez08,pampano20} and  \cite{gmm03,lopez08 hyp} respectively. Consider the following quantities:
\begin{align}
    &\Lambda_1=\beta-\gamma, &\Lambda_2=(\beta+\gamma)^2-4\alpha\delta.
\end{align}
When $\Lambda_1=0$ semi-quadratic surfaces become LW-surfaces. LW-surfaces are said to be \textit{hyperbolic} (\textit{elliptic}) when $\Lambda_2<0$ ($>0$) \cite{gmm03,lopez08 hyp} or satisfy $\Lambda_2=0$ which characterises tubular surfaces. This motivates the following nomenclature
\begin{Def}
    A semi-quadratic Weingarten surface satisfying $\Lambda_2>\Lambda_1^2$ is said to be \textit{elliptic}. If $\Lambda_2<\Lambda_1^2$ it is said to be \textit{hyperbolic}.
\end{Def}
The relative sizes and signs of $\Lambda_1$ and $\Lambda_2$ strongly control a semi-quadratic surface's behaviour:
\begin{Prop}\label{prop: qw surface lam_2>0}
A semi-quadratic Weingarten surface cannot have umbilic points unless $\Lambda_2\geq0$.
\end{Prop}
\begin{proof}
First assume $\alpha\neq0$, otherwise $\Lambda_2\geq 0$ and we are done. If we have such a surface, the curvatures at the umbilic point must satisfy the relation (\ref{eq: quadratic Weingarten relationship}) when $k_1=k_2=k$, i.e.
\[
\alpha k^2+(\beta+\gamma)k+\delta=0.
\]
Solving the above quadratic implies the curvatures at the umbilic point satisfy
\begin{equation}\label{eq: umbilic k quadratic weingarten}
 k=\frac{1}{2\alpha}\left(-(\beta+\gamma)\pm \sqrt{\Lambda_2}\right),
\end{equation}
and therefore $k\in\mathbb{R}$ if and only if $\Lambda_2\geq 0$.
\end{proof}
\begin{Prop}\label{prop: qw ellipticity condition}
 A semi-quadratic Weingarten surface's Weingarten relation is an elliptic PDE at an umbilic point if and only if
 \begin{equation}
     \Lambda_2>\Lambda_1^2.
 \end{equation}
\end{Prop}
\begin{proof}
Let $\mathcal{S}$ be a semi-quadratic Weingarten surface with an umbilic point $p$. Also let
\begin{equation}\label{eq: W of a semi QW surface}
W(k_1,k_2)=\alpha k_1k_2+\beta k_1+\gamma k_2+\delta.
\end{equation}
Let $k$ be the common value of $k_1(p)$ and $k_2(p)$, then $k$ is given by either one of the values in equation (\ref{eq: umbilic k quadratic weingarten}) and necessarily $\Lambda_2\geq0$. In which case one can calculate that
\begin{equation}\label{eq: ellipticity is lambda_2-lambda_1 squared}
    \left. \left(\pdv{W}{k_1} \cdot \pdv{W}{k_2}\right) \right|_p =\frac{1}{4}(\Lambda_2-\Lambda_1^2).
\end{equation}
Therefore the result follows by the definition of a Weingarten relation being elliptic (\ref{eq: ellipticity cond}).
\end{proof}
\begin{Prop}\label{prop: qw umbilic slope}
 If $\Lambda_2\neq \Lambda_1^2$, the umbilic slope at an isolated umbilic point $p$ of a semi-quadratic Weingarten surface takes one of the values
 \begin{equation}\label{eq: QW umbilic slope}
\mu_p=\frac{\Lambda_1\pm\sqrt{\Lambda_2}}{\Lambda_1\mp\sqrt{\Lambda_2}},
 \end{equation}
 If $\Lambda_2=\Lambda_1^2$, then at such a point $p$ either $\mu_p=0$ or $\mu_p$ is unbounded.
 \end{Prop}
\begin{proof}
 The possible values of $\mu_p$ are simply the possible slopes the algebraic curve $W(k_1,k_2)=\alpha k_1k_2+\beta k_1+\gamma k_2 +\delta=0$ intersects the diagonal line $k_1=k_2$. Let $k_0$ be given by equation (\ref{eq: umbilic k quadratic weingarten}) so that $W(k_0,k_0)=0$.  If $\Lambda_2\neq \Lambda_1^2$, then equation (\ref{eq: ellipticity is lambda_2-lambda_1 squared}) implies ${\pdv{W}{k_2}}\neq 0$. Then if $k_1=k_2=k_0$,
 \[
 \mu_p=-{\pdv{W}{k_1}}\bigg / {\pdv{W}{k_2}}=\frac{\Lambda_1\pm\sqrt{\Lambda_2}}{\Lambda_1\mp\sqrt{\Lambda_2}}.
 \]
If $\Lambda_2=\Lambda_1^2$ then  $\beta\gamma=\alpha\delta$, forcing $W^{-1}\{0\}$ to be either a line of constant $k_1$, a line of contant $k_2$, or a union of the two. In which case $\mu_p$ may be either $0$ or unbounded.
\end{proof}
We will study $\slr$ transformations between semi-quadratic Weingarten surfaces that are rotationally symmetric and satisfy $\Lambda_2>0$. The case $\Lambda_2<0$ has been studied in the particular scenario $\Lambda_1=0$ with a classification result being obtained in the rotationally symmetric setting \cite{lopez08 hyp}. To make the following discussion simpler we assume w.l.o.g. that $\Lambda_2=1$ by dividing equation (\ref{eq: quadratic Weingarten relationship}) through by $\sqrt{\Lambda_2}$. In such a case we say relationship (\ref{eq: quadratic Weingarten relationship}) is normalised.
\begin{Prop}\label{prop: slr preserve qw}
The $\slr$ transformations map normalised quadratic Weingarten relations to normalised quadratic Weingarten relations. If the initial relation has coefficients $\alpha,\beta,\gamma,\delta \in \mathbb{R}$, the coefficients of the target relationship are given by
\begin{align}\label{eq: alpha' equ.}
    &\alpha'=\alpha d^2 + \delta b^2 + (\beta+\gamma)bd,\\\label{eq: beta' equ.}
    &\beta'=\alpha cd+\delta ab +(\beta+\gamma)bc +\beta,\\\label{eq: gamma' equ.}
    &\gamma'=\alpha cd+\delta ab +(\beta+\gamma)bc +\gamma,\\
    & \delta'=\alpha c^2 + \delta a^2 + (\beta+\gamma)ac.
    \label{eq: delta' equ.}
\end{align}
Furthermore $\Lambda_1^2$ is an invariant. 
\end{Prop}
\begin{proof}
We substitute the $\slr$ transformations from Proposition \ref{prop: coordinate transform} in the form
\[
k_i\mapsto \frac{c+dk_i}{a+bk_i},
\]
into the initial quadratic relationship (\ref{eq: quadratic Weingarten relationship}). If $b\neq0$, multiplying through by any denominators gives a relationship of the form
\begin{equation}\label{eq: target relation}
\alpha' k_1k_2+\beta' k_1 + \gamma' k_2 +\delta' =0,
\end{equation}
with $\alpha',\beta,'\gamma'$ and $\delta'$ satisfying equations (\ref{eq: alpha' equ.})-(\ref{eq: delta' equ.}). It is immediate from these equations that $\Lambda_1'^2=\Lambda_1^2$. Furthermore, a calculation verifies  that $\Lambda_2'=\Lambda_2=1$ and so the target relationship (\ref{eq: target relation}) is normalised. If $b=0$, the target relation is already of the form (\ref{eq: target relation}), however with $\Lambda_1'=\frac{d}{a}\Lambda_1$ and $\Lambda_2'=\frac{d^2}{a^2}$. Dividing equation (\ref{eq: target relation}) through by $d/a$ normalises the relationship so that $\Lambda_2'=1$ implying $\Lambda_1'^2=\Lambda_1^2$.
\end{proof}
\begin{remark}
When the semi-quadratic relationship is not normalised, the corresponding invariant is $\Lambda_1^2/\Lambda_2$, shown by multiplying through a normalised equation by $\sqrt{\Lambda_2}$.
\end{remark}
\begin{Prop}\label{prop: sl2 acts transitively on QW}
 When $\Lambda_2>0$, the $\slr$ transformations act transitively on the families of normalised semi-quadratic Weingarten relations with the same $\Lambda_1^2$.
\end{Prop}
\begin{proof} 
Fix the value of $\Lambda_1$ and let $\Lambda_2=1$. Note that by re-arrangement of equation (\ref{eq: quadratic Weingarten relationship}), $\Lambda_1$ can always be assumed positive and therefore it can be assumed that $\Lambda_1$ is conserved under the $\slr$ transformations rather than $\Lambda_1^2$. Let $(\alpha,\beta,\gamma,\delta)$ be the parameters of the initial relation, and $(\alpha',\beta',\gamma',\delta')$ be the parameters of the target relation. To prove transitivity one just needs to show that the system of equations (\ref{eq: alpha' equ.})-(\ref{eq: delta' equ.}) has a solution $(a,b,c,d)\in \mathbb{R}^4$, with $ad-bc=1$, for any choice of the two 4-tuples $(\alpha,\beta,\gamma,\delta),(\alpha',\beta',\gamma',\delta')\in\mathbb{R}^4$ satisfying $\Lambda_1=\Lambda_1'$ and $\Lambda_2=\Lambda'_2=1$. First substitute $\Lambda_1=\beta'-\gamma'=\beta-\gamma$ 
into equations (\ref{eq: alpha' equ.})-(\ref{eq: delta' equ.}) to remove the parameters $\beta$ and $\beta'$. If one considers $\slr$ transformations for which $c\neq0$, the constraint $ad-bc=1$ can be used to eliminate $b$ giving a system of equations relating the initial coefficients with the target coefficients:
\begin{align} \label{eq: alpha' with invar.}
    &\alpha'=\frac{1}{c^2}\left[\alpha c^2d^2+\delta(ad-1)^2+(\Lambda_1+2\gamma)cd(ad-1)\right],\\\label{eq: gamma' with invar.}
    &c\gamma'=\alpha c^2d+\left(\delta a+(\Lambda_1+2\gamma)c\right)(ad-1)+\gamma c,\\ \label{eq: delta' with invar.}
    &\delta'=\alpha c^2 + \delta a^2+(\Lambda_1+2\gamma)ac.
\end{align}
 The equations (\ref{eq: alpha' with invar.}), (\ref{eq: gamma' with invar.}) and (\ref{eq: delta' with invar.}) correspond to (\ref{eq: alpha' equ.}), (\ref{eq: beta' equ.}-\ref{eq: gamma' equ.}), and (\ref{eq: delta' equ.}) respectively. Transitivity is proven by considering 2 separate cases.
{\flushleft \bf \underline{Case: $\delta' \neq 0$.}}\hfill\\

 Taking $c\neq0$, solving the equation $\Lambda_2=\Lambda_2'$ for $\alpha'$ gives
 \[
\alpha'=\frac{(\Lambda_1+2\gamma')^2-(\Lambda_1+2\gamma)^2+4\alpha\delta}{4\delta'},
\]
Substituting equations (\ref{eq: gamma' with invar.}) and (\ref{eq: delta' with invar.}) into the above implies equation (\ref{eq: alpha' with invar.}), hence equations (\ref{eq: gamma' with invar.}) and (\ref{eq: delta' with invar.}) form an under-determined system
\begin{align}
\label{eq: gamma' and delta'}
    &c\gamma'=\delta'd-\delta a -(\Lambda_1+\gamma)c,
    &\delta'=\alpha c^2 + \delta a^2+(\Lambda_1+2\gamma)ac,
\end{align}
where we have re-written the equation for $\gamma'$, removing the $\alpha c^2d$ term by virtue of equation (\ref{eq: delta' with invar.}) to make calculations easier. This system can be solved to give $a$ and $d$ in terms of $c$: If $\delta \neq 0$, then the solutions are given by
\begin{align*}
    &a=\frac{1}{2\delta}\left(-(\Lambda_1+2\gamma)c\pm \sqrt{c^2+4\delta\delta'}\right), & d=\frac{1}{2\delta'}\left((\Lambda_1+2\gamma')c\pm \sqrt{c^2+4\delta\delta'}\right).
\end{align*}
The parameters $a$ and $d$ can always be taken to be real by taking a sufficiently large $c$, and $b$ is determined by $ad-bc=1$. If $\delta=0$, then the solutions are
\begin{align*}
    &a=\frac{\delta'-\alpha c^2}{c(\Lambda_1+2\gamma)}, & d=\frac{c}{\delta'}(\Lambda_1+\gamma+\gamma'),
\end{align*}
noting that when $\delta=0$, $(\Lambda_1+2\gamma)^2=(\beta+\gamma)^2=\Lambda_2\neq0$. Hence we have found an $\slr$ transformation taking $(\alpha,\beta,\gamma,\delta)$ to $(\alpha',\beta',\gamma',\delta')$ when $\delta'\neq0$.
{\flushleft \bf \underline{Case: $\delta'=0$.}}\hfill\\

Note that if $\delta \neq 0$, by the above case there exists an $\slr$ transformation sending $(\alpha',\beta',\gamma',0)$ to $(\alpha,\beta,\gamma,\delta)$. Taking the inverse transformation proves transitivity when $\delta \neq 0$. Now assume $\delta=0$. Any normalised relationship for which $\delta=0$ must be of the form
\[
\alpha k_1k_2 + \frac{1}{2}(\Lambda_1\pm1)k_1 +\frac{1}{2}(-\Lambda_1\pm1)k_2=0,
\]
since $\beta$ and $\gamma$ must solve $\beta-\gamma=\Lambda_1$ and $(\beta+\gamma)^2=\Lambda_2=1$ simultaneously. Thus $\delta=\delta'=0$ implies that $\gamma$ and $\gamma'$ take either of the values $\frac{1}{2}(-\Lambda_1\pm 1)$ and so for this sub-case, all initial and target relations must be of the respective forms
\[
\left(\alpha,\frac{1}{2}(\Lambda_1\pm 1),\frac{1}{2}(-\Lambda_1\pm 1),0\right), \qquad \left(\alpha',\frac{1}{2}(\Lambda_1\pm 1),\frac{1}{2}(-\Lambda_1\pm 1),0\right).
\] If an $\slr$ transformation is taken with $c\neq0$, equations (\ref{eq: alpha' with invar.}) and (\ref{eq: delta' with invar.}) are solved for $a$ and $d$ to give
\begin{align}
&a=-\frac{\alpha c}{\Lambda_1+2\gamma}, &d=-\frac{\alpha' c}{\Lambda_1+2\gamma}.
\end{align}
The remaining equation (\ref{eq: gamma' with invar.}) implies that $\gamma=\frac{1}{2}(-\Lambda_1\pm 1)$ and $\gamma'=\frac{1}{2}(-\Lambda_1\mp 1)$.
On the other-hand, taking $c=0$ fixes $\gamma=\gamma'$ which can be seen by equation (\ref{eq: gamma' equ.}). Equation (\ref{eq: alpha' equ.}) can then be solved for $d$ in terms of $b$, if $\alpha \neq 0$, then; 
\begin{align*}
    &d=\frac{1}{2\alpha}\left(-(\Lambda_1+2\gamma)b\pm \sqrt{b^2+4\alpha\alpha'}\right),
\end{align*}
taking $b$ sufficiently large implies $d\in\mathbb{R}$, $a$ is then determined by $ad-bc=1$. Alternatively, if $\alpha=0$ then solving equation (\ref{eq: alpha' equ.}) gives
\[
d=\frac{\alpha'}{b(\Lambda_1+2\gamma)}.
\]
Therefore we have found $\slr$ transformations taking $$\left(\alpha,\frac{1}{2}(\Lambda_1\pm 1),\frac{1}{2}(-\Lambda_1\pm 1),\delta \right) \mapsto \left(\alpha',\frac{1}{2}(\Lambda_1\mp 1),\frac{1}{2}(-\Lambda_1\mp 1),\delta'\right),$$
if $c \neq 0$, or
$$\left(\alpha,\frac{1}{2}(\Lambda_1\pm 1),\frac{1}{2}(-\Lambda_1\pm 1),\delta \right) \mapsto \left(\alpha',\frac{1}{2}(\Lambda_1\pm 1),\frac{1}{2}(-\Lambda_1\pm 1),\delta'\right),$$
if $c=0$, which together describe all possible transformations between relations of type $\delta'=\delta=0$.\\

Since any pair of initial and target relationships fall into one of the above cases, transitivity has been proven.
\end{proof}
\vspace{0.1in}
The above result is useful in classifying semi-quadratic Weingarten surfaces based on their $\Lambda_1$ and $\Lambda_2$ values. We first consider a special case.
\vspace{0.1in}
\begin{Thm}\label{thm: parabolic QW are canal}
    Let $\mathcal{S}$ be a connected rotationally symmetric semi-quadratic Weingarten surface for which $\Lambda_1^2=\Lambda_2$. Then $\mathcal{S}$ is a subset of a torus of revolution, round sphere, plane, cone or cylinder.
\end{Thm}
\begin{proof}
    Note that since $\Lambda_2=\Lambda_1^2\geq0$, either $\Lambda_2=0$ or $\Lambda_2>0$. Taking $\Lambda_2=0$ implies $\Lambda_1=0$, which gives a LW relationship:
    \[
\lambda K+\Upsilon H+C=0,
    \]
for some constants $\lambda,\Upsilon,C\in\mathbb{R}$. It has been remarked in \cite{lopez08 hyp} that LW relationships for which
\[
\Upsilon^2-4\lambda C=0
\]
characterise either tubular surfaces or planes (that is in the rotational case, cylinders, tori of revolution or planes). It is observed that for LW surfaces, $\Lambda_2=\Upsilon^2-4\lambda C$ and so the $\Lambda_2=0$ case has already been proven. Now consider the $\Lambda_2>0$ case.\\

We first recall the following fact: Every connected component of a surface with a constant principal curvature is a subset of either a round sphere, a tube over a curve (if the constant principal curvature is non-zero)\cite{st70} or a developable surface (if a principal curvature is constantly zero). A connected surface of revolution with constant principal curvatures must therefore be a subset of a torus of revolution or a round sphere (if $K$ is not identically zero) or  a subset of a plane, cone or cylinder (if $K\equiv 0$).\\

Our strategy of proof will be to, assuming $\mathcal{S}$ is semi-quadratic, show that one of the principal curvatures of $\mathcal{S}$ is constant. We first assume $K$ is nowhere vanishing on $\mathcal{S}$, i.e. $\mathcal{S}$ is non-flat. Assume for contradiction that both $k_1$ and $k_2$ are non constant. Take a $\slr$ transformation of $\mathcal{S}$, whose relation we may assume w.l.o.g. to be normalised, to another semi-quadratic surface $\widetilde{\mathcal{S}}$ sharing the same invariant $\Lambda_1$ and satisfying either of the Weingarten relations
    \begin{equation}\label{eq: transformed equ canal.}
     \frac{1}{2}\left(\Lambda_1\pm 1\right)\widetilde k_1+\frac{1}{2}\left(-\Lambda_1\pm 1\right)\widetilde k_2=0,
 \end{equation}
 where the principal curvatures $\widetilde k_i$ of $\widetilde{\mathcal{S}}$ are related to the principal curvatures $k_i$ of $\mathcal{S}$ by
 \begin{align*}
&\widetilde k_i=\frac{c+dk_i}{a+bk_i}, &\text{for some }\begin{pmatrix}
    a & b \\ c & d
\end{pmatrix}\in \slr,
 \end{align*}
 $i=1,2$. Note that by assumption $k_1$ and $k_2$ are not identically equal to $-a/b$, hence equation (\ref{eq: transformed equ canal.}) is satisfied on all finite parts of $\mathfrak{R}(\widetilde{\mathcal{S}})$. Since $\Lambda_1^2=1$ however, equation (\ref{eq: transformed equ canal.}) implies at least one of the principal curvatures of $\widetilde{\mathcal{S}}$ must be identically zero, implying one of the principal curvatures of $\mathcal{S}$ takes the constant value $-c/d$ which is a contradiction. Now consider the more general case that $K$ vanishes on parts of $\mathcal{S}$ and let
 \begin{align*}
&\mathcal{S}_0={K}^{-1}\{0\}, &\mathcal{S}_\pm={K}^{-1}(\mathbb{R}\backslash\{0\}),
 \end{align*}
 so that $\mathcal{S}$ can be partitioned as $\mathcal{S}=\mathcal{S}_0\cup \mathcal{S}_\pm$. We will show that $\mathcal{S}_0\neq\emptyset$ and $\mathcal{S}_\pm\neq\emptyset$ cannot hold simultaneously. Assume otherwise, then since $\mathcal{S}_\pm$ is open it is a sub-surface so has a countable number of connected components which we denote as $V_n$, $n\in\mathbb{N}$. Each $V_n$ is connected and non-flat and so the repeating the argument previously given shows that $k_2$ takes a constant value on each $V_n$, denoted by $c_n$. Therefore
\begin{align*}
k_2(\mathcal{S})=k_2(\mathcal{S}_0)\cup k_2(\mathcal{S}_\pm)=k_2(\mathcal{S}_0)\cup  \bigcup_{n=1}^\infty k_2(V_n)=\{0\}\cup \bigcup_{n=1}^\infty \{c_n\},
\end{align*}
implying $k_2(\mathcal{S})$ is disconnected, contradicting the connectedness of $\mathcal{S}$ or the continuity of $k_2$. Hence one of $\mathring{\mathcal{S}}_0$ or $\mathcal{S}_\pm$ is empty. The $\mathcal{S}_0=\emptyset$ case was considered above. On the other hand if $\mathcal{S}_\pm=\emptyset$ then $K\equiv 0$ on $\mathcal{S}$ and $\mathcal{S}$ is a plane, cone or cylinder.

\end{proof}
\vspace{0.1in}
The main theorem of this section is motivated by the classification result given in \cite{pampano20} for surfaces of revolution satisfying the relation
\begin{equation}\label{eq: pure k linear.}
    k_2=\lambda k_1.
\end{equation}
\begin{Thm}[\cite{pampano20}, Theorem 4.1, page 15]\label{thm: pampano result}\hfill\newline
 \begin{enumerate}
     \item If $\lambda>0$, the surfaces satisfying (\ref{eq: pure k linear.}) are either planes, closed surfaces with a convex profile curve, or subsets thereof. In the special case $\lambda=1$, the only solutions are round spheres or subsets thereof.
     \item If $\lambda<0$, the surfaces satisfying (\ref{eq: pure k linear.}) are either planes or open, catenoid-like surfaces with a convex profile curve, or subsets thereof.
 \end{enumerate}
 \end{Thm}
 \vspace{0.1in}
 In the case $\lambda=0$, the above Weingarten surfaces are developable, i.e. planes, cones or cylinders. Surfaces satisfying equation (\ref{eq: pure k linear.}) are related in $\mathbb{E}^3$ as follows.
\vspace{0.1in}
\begin{Lem}\label{prop: homothety relates pure k lin}
 All rotationally symmetric, non-planar, surfaces satisfying the relation (\ref{eq: pure k linear.}) for a fixed $\lambda \neq 0$ are related by a homothety.
\end{Lem}
\begin{proof}
The radii of curvature of a non-planar solution of relation (\ref{eq: pure k linear.}) may be assumed to be finite as $\lambda \neq 0$. Thus, the $r_1$ radius of curvature of such a surface is given by equation (\ref{eq: hopf rad of curv}) in \Cref{ex: hopf sphere r1} with a slight change in notation $\lambda \mapsto 1/\lambda$ and by setting $C=0$.
\begin{align}\label{eq: pure k linear r_1}
    &r_1(\theta)=\frac{\lambda A_0 \sin^{\frac{1}{\lambda}-1}\theta}{\lambda-1}, & \lambda \neq 0,1.\\
    &r_1(\theta)=\text{ constant}, &\lambda=1, \lambda \neq 0.
\end{align}
where $A_0$ is a constant of integration and has the effect of scaling $r_1$. Since $r_2=\frac{1}{\lambda}r_1$, the curvatures of every possible surface of revolution satisfying relation (\ref{eq: pure k linear.}) for a fixed $\lambda\neq0$ are related by a map
\[
(r_1,r_2)\mapsto(Cr_1,Cr_2),
\]
$C\in\mathbb{R}\backslash\{0\}$. Therefore such surfaces are related by a homothety.
\end{proof}
\vspace{0.1in}
\begin{Cor}\label{cor: hopf spheres are par trans and homo of unit}
    Any two rotationally symmetric linear Hopf surfaces satisfying a given relation (\ref{eq: lin hopf surf}) may be related to each other geometrically by conjugating a homothety $h$ with a parallel translation $p$, i.e. by the map $p\circ h \circ p^{-1}$.
\end{Cor}
\begin{proof}
    Given any two linear Hopf surfaces $S_1$ and $S_2$ satisfying
    \[
    r_2=\lambda r_1+C,
    \]
after applying a suitable parallel translation denoted $p$, their Weingarten relation is mapped to relation (\ref{eq: pure k linear.}). Let $\hat{S}_1$ and $\hat{S}_2$ be the images of $S_1$ and $S_2$ under $p$. Proposition \ref{prop: homothety relates pure k lin} implies $\hat{S}_1$ and $\hat{S}_2$ are related by a homothety $h$, hence $S_1$ and $S_2$ are related by
     \[
     S_1 \xrightarrow{p} \hat{S}_1 \xrightarrow{h} \hat{S}_2 \xrightarrow{p^{-1}} S_2,
     \]
$p^{-1}$ being a parallel translation.
\end{proof}
\vspace{0.1in}
Corollary \ref{cor: hopf spheres are par trans and homo of unit} says that any rotationally symmetric linear Hopf surface satisfying a fixed relation is a composition of a parallel translation and a homothety of a single, surface of revolution satisfying relation (\ref{eq: pure k linear.}). The following theorem generalises this idea to semi-quadratic Weingarten surfaces.

\vspace{0.1in}
\begin{Thm}\label{thm: slr on QW to pure lin}
Any non-flat rotationally symmetric, connected semi-quadratic Weingarten surface with $\Lambda_2>0$ is the image under a composition of homotheties, parallel translations and reciprocal transformations of a Weingarten surface satisfying the relation
\begin{equation*}
k_2=\lambda k_1,
\end{equation*}
for $\lambda>0$ when the surface is elliptic, or for $\lambda<0$ when the surface is hyperbolic.
\end{Thm}
\begin{proof}
Let $\mathcal{S}$ be as such with invariant $\Lambda_1^2$ and assume $\Lambda_2=1$. In the case that $\Lambda_1^2=1$, \Cref{thm: parabolic QW are canal} implies that $\mathcal{S}$ is a round sphere, torus of revolution, plane, cone or cylinder. Of these, only the round sphere is non-flat, in which case relationship (\ref{eq: pure k linear.}) is already satisfied for $\lambda=1$ so the required $\slr$ transformation can be taken to be the identity, hence the result is proven when $\Lambda_1^2=1$. Now assume $\Lambda_1^2 \neq 1$. By Proposition \ref{prop: sl2 acts transitively on QW} we can map the Weingarten relation of $\mathcal{S}$ to either of the relations
 \begin{equation}\label{eq: pure k lin in QW form.}
     \frac{1}{2}\left(\Lambda_1\pm 1\right)k_1+\frac{1}{2}\left(-\Lambda_1\pm 1\right)k_2=0,
 \end{equation}
with a $\slr$ transformation. 
The relations (\ref{eq: pure k lin in QW form.}) are equivalent to 
\begin{equation}\label{eq: pure k lin mu pm}
k_2=\lambda^{\pm}k_1, \qquad \lambda^\pm=\frac{\Lambda_1\pm 1}{\Lambda_1\mp 1}.
\end{equation}
Hence $\mathcal{S}$ can be mapped by a composition of parallel translations, a homothety and a reciprocal transformation to either of two surfaces of revolution satisfying the Weingarten relation (\ref{eq: pure k lin mu pm}) with $\lambda=\lambda^+$ or $\lambda=\lambda^-$. If $\mathcal{S}$ is hyperbolic so that $\Lambda_1^2>1$ then $\lambda^\pm>0$. On the other hand if $\mathcal{S}$ is elliptic then $\Lambda_1^2<1$ and $\lambda^\pm<0$, which is as expected since $\slr$ transformations preserve ellipticity.
\end{proof}
\vspace{0.1in}
\begin{remark}
Since $\lambda^+=1/\lambda^-$, both of the possible classes of target surfaces given in Theorem \ref{thm: slr on QW to pure lin} are reciprocal transformations of one another given by Proposition \ref{prop: reciprocal trm} with $\mathcal{A}=1$.
\end{remark}
\vspace{0.1in}
\noindent{\bf Statements and Declarations}:

The second author was supported by the Institute of Technology, Tralee / Munster Technological University Postgraduate Scholarship Programme. 
The authors have no other relevant financial or non-financial interests to disclose. No data was collected in the course of this research.
\vspace{0.2in}


\end{document}